\documentclass[12pt,reqno]{amsart}
\numberwithin{equation}{section}

\usepackage{mathrsfs}
\usepackage{times}
\usepackage{amsfonts}
\usepackage{amssymb}
\usepackage{amsmath}
\usepackage{wasysym}
\usepackage{enumerate}
\usepackage[latin1]{inputenc}
\usepackage[usenames,dvipsnames]{color}
\usepackage{graphicx}

\newcommand{\ring}[1]{\ensuremath{\mathbb{#1}}}
\newcommand\minus{\smallsetminus}

\newcommand\<{\langle}
\renewcommand\>{\rangle}
\newcommand\CC{\ring{C}}

\newcommand\NN{\ring{N}}

\newcommand\QQ{\ring{Q}}
\newcommand\RR{\ring{R}}

\newcommand\VV{{\mathscr V}}
\newcommand\ZZ{\ring{Z}}

\newcommand\cB{{\mathscr B}}
\newcommand\cC{{\mathscr C}}

\newcommand\cK{{\mathscr K}}
\newcommand\cN{{\mathscr N}}
\newcommand\cS{{\mathscr S}}
\newcommand\cT{{\mathscr T}}
\newcommand\cU{{\mathscr U}}

\newcommand\del{\partial}

\newcommand{\HA}{H_A(\beta)}
\newcommand{\HrA}{H_{\rho(A)}(\beta_0,\beta)}

\newcommand{\be}{{\boldsymbol{\varepsilon}}}

\newcommand\vol{{\rm vol}}

\newcommand\rank{{\rm rank}}
\newcommand\conv{{\rm conv}}
\newcommand\tp{{\rm top}}
\newcommand\supp{{\rm supp}}

\newcommand\nsup{{\rm nsupp}}
\newcommand\nsupp{{\rm nsupp}}

\newcommand{\inw}{{\rm in}_w}
\newcommand{\inow}{{\rm in}_{(0,w)}}
\newcommand{\inww}{{\rm in}_{(-w,w)}}

\newcommand{\Fw}{{\mathscr{N}}_w}

\newtheorem{theorem}{Theorem}[section]
\newtheorem{lemma}[theorem]{Lemma}

\newtheorem{corollary}[theorem]{Corollary}

\newtheorem{proposition}[theorem]{Proposition}
\newtheorem{definition}[theorem]{Definition}
\newtheorem{convention}[theorem]{Convention}
\newtheorem{notation}[theorem]{Notation}
\theoremstyle{definition}

\newtheorem{remark}[theorem]{Remark}

\newtheorem{example}[theorem]{Example}

\begin{document}

\title{Nilsson solutions for irregular $A$-hypergeometric systems}

\author{Alicia Dickenstein}
\address[AD]{Departamento de Matem\'atica,
FCEN, Univ. de Buenos Aires, Buenos Aires, Argentina.}
\email{alidick@dm.uba.ar}
\thanks{AD was 
partially supported by UBACYT 20020100100242, CONICET PIP 112-200801-00483 
and ANPCyT PICT 2008-0902, Argentina}

\author{Federico Nicol\'as Mart\'{\i}nez}
\address[FM]
{Departamento de Matem\'atica, 
FCEN, Univ. de Buenos Aires,  Buenos Aires, Argentina.}
\email{fnmartin@dm.uba.ar}
\thanks{FNM was supported by a doctoral grant from CONICET}

\author{Laura Felicia Matusevich}
\address[LFM]{Department of Mathematics,
Texas A\&M University, College Station, TX, USA.}
\email{laura@math.tamu.edu}
\thanks{LFM was partially supported by NSF Grant DMS 0703866 
and a Sloan Research Fellowship}

%
%

\begin{abstract}
We study the solutions of irregular $A$-hyper\-geo\-metric systems
that are cons\-truc\-ted from Gr\"obner degenerations 
with respect to 
generic positive weight vectors. These are 
formal logarithmic Puiseux series that belong to 
explicitly described Nilsson rings, and are therefore
called (formal) Nilsson series. When the weight vector is a
perturbation of $(1,\dots,1)$, these series converge and provide a
basis for the (multivalued) holomorphic hyper\-geo\-metric functions in a specific open
subset of $\CC^n$. Our results are more explicit when the parameters
are generic or when the solutions studied are logarithm-free.
We also give an alternative proof of a result of
Schulze and Walther that inhomogeneous $A$-hyper\-geo\-metric systems have
irregular singularities. 
\end{abstract}



\maketitle

\pagestyle{myheadings}
\markleft{Alicia Dickenstein, Federico Nicol\'as Mart{\'{\i}}nez 
and Laura Felicia Matusevich}

%
%

\section{Introduction}
\label{sec:intro}

The Frobenius method is a symbolic procedure for 
solving a linear ordinary differential 
equation in a neighborhood of a regular singular point. 
The solutions are represented as convergent
logarithmic Puiseux series that belong to the Nilsson class.
In the multivariate case, a direct generalization of the Frobenius
method, called the \emph{canonical series algorithm}, was introduced
by Saito, Sturmfels and Takayama in Chapter~2.5 of~\cite{SST}, based on
Gr\"obner degenerations in the Weyl algebra $D$. When applied
to a regular holonomic left $D$-ideal, 
this procedure yields a basis of the
solution space. The basis elements belong to an explicitly described
Nilsson ring, and are therefore called
\emph{Nilsson series}, or \emph{Nilsson solutions}.
Each Nilsson ring is constructed using a weight vector; the
choice of weight vector is a way of determining the common
domain of convergence of the corresponding solutions. 
The canonical series procedure requires a regular holonomic input; 
although one can run this algorithm on
holonomic left $D$-ideals that have irregular
singularities, there is no guarantee that the output 
series converge, or even that the  correct number of basis elements
will be produced. 

In this article, we study the solutions of 
the $A$-hyper\-geo\-metric systems introduced by
Gelfand, Graev, Kapranov and Zelevinsky~\cite{GGZ,GKZ}. 

\begin{definition}\label{def:A-hyp}
Denote by $D$ the Weyl algebra on $x_1,\dots,x_n$ and $\del_1,\dots,\del_n$,
where $\del_i$ stands for the partial derivative with respect to $x_i$.
Let $A = [a_{ij}] \in \ZZ^{d \times n}$ whose rows $\ZZ$-span $\ZZ^d$, 
and let $\beta \in \CC^d$.
The \textbf{$A$-hyper\-geo\-metric system with parameter $\beta$} is the left
$D$-ideal 
\[
H_A(\beta) =  I_A  + 
   \< E_1 - \beta_1,\dots,E_d -\beta_d\> \subset D,
\]
where $E_i = \sum_{j=1}^n a_{ij} x_j \del_j$, $1 \leq i \leq d$, and $I_A$
denotes the \textbf{toric ideal}
\[
I_A = \< \del^u - \del^v \mid A\cdot u = A\cdot v\> \subseteq \CC[\del].
\]
\end{definition}

These left $D$-ideals are always holonomic~\cite{Ado, GKZ}.
It is known that $\HA$ is regular holonomic
if and only if the $\QQ$-rowspan of the matrix $A$ contains the vector
$(1,\dots, 1)$.
The if direction was proved by Hotta in his work on equivariant $D$-modules~\cite{H}; 
Saito, Sturmfels and Takayama gave a partial converse, see
Theorem~2.4.11 in~\cite{SST}, where it is assumed that the parameter $\beta$ is
generic. Their proof uses Gr\"obner methods in $D$.

A different strategy 
to show that a $D$-ideal is not regular holonomic
is to prove that it has \emph{slopes}.
The analytic slopes of a
$D$-module were introduced by Meb\-khout in~\cite{mebkhout}, while an
algebraic version was given by Laurent~\cite{laurent}. These authors
have shown that the analytic and algebraic slopes of a $D$-module
along a smooth hypersurface agree~\cite{laurent-mebkhout}. From a
computational perspective, Assi, Castro--Jim\'enez and Granger gave a
Gr\"obner basis algorithm to find algebraic slopes~\cite{GBforslopes}.
There has been an effort to compute the (algebraic) slopes of $\HA$ along a
coordinate hypersurface.
In the cases $d=1$ and $n-d=1$, these slopes were determined
by Castro--Jim\'enez and Takayama~\cite{CT}, and 
Hartillo--Hermoso~\cite{Hartillo,Hartillo2}. More generally,
Schulze and Walther~\cite{SW} have calculated the slopes of $\HA$
under the assumption that the cone over the columns of $A$
contains no lines. Such a cone is called \emph{strongly convex}. The
fact that slopes of $\HA$ always exist when the vector $(1, \dots, 1)$ does not
belong to the rowspan of $A$, implies that $\HA$ has irregular singularities.
Thus, Corollary~3.16 in~\cite{SW} gives a converse for Hotta's regularity
theorem in the strongly convex case. We give an alternative proof of
this converse here, 
by extending ideas of Saito, Sturmfels and Takayama.
The main technical obstacle to overcome
is the potential existence of logarithmic hyper\-geo\-metric series.

Further insight into the solutions of hyper\-geo\-metric system comes
from the analytic approach taken up by Castro--Jim\'enez and
Fern\'andez--Fern\'andez~\cite{CF,FF}, who studied the Gevrey
filtration on the 
irregularity complex of an $A$-hy\-per\-geo\-me\-tric system. Since
formal series solutions of irregular systems need not converge, a
study of the Gevrey filtration provides information on how far such
series are from convergence.

Even though the regularity of $\HA$ is independent of $\beta$,
in practice, assuming that the parameters are generic makes a difference. 
In this case, Ohara and Takayama~\cite{OT}
show that the method of canonical series 
for a weight vector which is a perturbation of $(1,\dots,1)$ 
produces a basis for the solution space of $\HA$ 
consisting of (convergent) Nilsson series that contain no logarithms.
In that work (and many others, such as~\cite{Ado, MMW, SST})
an important role is played by a related hyper\-geo\-metric
system that does have regular singularities, called the 
\emph{homogenized system}. One of our main goals is to precisely explain
the relationship between the solutions of $\HA$ and those of its
homogenization.

In order to produce a basis of solutions of
$\HA$ when $\beta$ is not
generic, logarithmic series cannot be avoided, even in
the regular case. Dealing with logarithmic solutions of $\HA$
poses technical challenges that we resolve here, allowing us
to lift the genericity hypotheses from the results of Ohara and Takayama:
running the canonical series algorithm on $\HA$ with weight 
vector (a perturbation of) $(1,\dots,1)$ always produces a basis of
(convergent) Nilsson solutions of $\HA$, if the cone
spanned by the columns of $A$ 
is strongly convex. 
On the other hand, formal solutions of irregular hyper\-geo\-metric
systems that are not 
Nilsson series need to be considered, even in one variable (see, for
instance,~\cite{cope}).

This article is organized as follows. In Section~\ref{sec:prelim} we
introduce the formal basic Nilsson solutions of
$\HA$ in the direction of a weight vector $w$
(Definition~\ref{defformal}), whose linear span is denoted by $\Fw(\HA)$.
Section~\ref{sec:invert-der} explains the relationship between 
$\Fw(\HA)$ and the solution space of an associated regular holonomic
hyper\-geo\-metric system. A linear map $\rho$ between these spaces is
constructed in  
Definition~\ref{def:homog-sol} and
Proposition~\ref{propo:Nilsson-to-Nilsson}. The main result in
Section~\ref{sec:invert-der} is Theorem~\ref{thm:image of rho}, which
states that $\rho$ is injective, and describes its image. 

In Section~\ref{sec:generic-parameters} we restrict
our attention to generic parameter vectors $\beta$,
and give a
combinatorial formula for the dimension of $\Fw(\HA)$ 
(Theorem~\ref{coro:dimnilsson}). This formula is reminiscent of the
multiplicities 
of the $L$-characteristic cycles of the $A$-hyper\-geo\-metric $D$-module
$D/\HA$ computed in~\cite{SW}, but note that these authors use
different weights in the Weyl algebra than we do: while we focus on
weight vectors of the form $(-w,w)$, Schulze and Walther work with
$(u,v)$ such that $u+v$ is a positive multiple of $(1,\dots,1)$.

In Section~\ref{sec:logarithm free} 
we lift the genericity assumption from $\beta$,
but study only $A$-hyper\-geo\-metric Nilsson series that contain no
logarithms, and show that they arise from logarithm-free solutions of
the associated regular holonomic hyper\-geo\-metric system
(Theorem~\ref{teo:map-of-fake-exps}). 
We show in Theorem~\ref{prop:bSaito} that the logarithm-free
basic Nilsson solutions of $\HA$ in the direction of a weight vector
$w$ span the vector space of formal logarithm-free $A$-hyper\-geo\-metric
series in the direction of $w$, proving the formula 
in Display~(7) in~\cite{S}.  

In Section~\ref{sec:convergence} we
analyze the convergence of formal Nilsson
solutions of $A$-hyper\-geo\-metric systems. The main result of that section
is Theorem~\ref{ultimoteorema}, which provides an explicit
construction for a basis of the space of (multivalued) holomorphic
solutions of $\HA$ that 
converge in a specific fixed open set, assuming that the cone spanned
by the columns of $A$ 
is strongly convex. While this was well known in the
regular case, it is a new result in general. 
Finally, under the same strong convexity assumption for $A$, we
use our study of formal $A$-hyper\-geo\-metric Nilsson series in
Section~\ref{sec:Hotta} to (re)prove
that $A$-hyper\-geo\-metric systems arising from inhomogeneous toric
ideals have irregular singularities 
for all parameters (Theorem~\ref{coro:lower rank initial}).

\subsection*{Acknowledgements}

We are very grateful to Christine Berkesch, who proved
Theorem~\ref{thm:rank-lifting} at our request.
Part of this work was
performed while the third author was visiting MSRI in the Fall 2009,
and while the 
three authors met at the National Algebra Meeting in La Falda,
Argentina, in August 2008. We thank the organizers of these events for the
wonderful research atmosphere and hospitality. We are very grateful to
the anonymous referee, whose thoughtful suggestions have improved this article.

\section{Initial ideals and formal Nilsson series}
\label{sec:prelim}

Throughout this article, $A$ is a fixed $d \times n$ integer matrix
whose columns $\ZZ$-span the lattice $\ZZ^d$. 
Given $\beta \in \CC^d$, we consider the $A$-hyper\-geo\-metric system
$\HA$ introduced in Definition~\ref{def:A-hyp}, and
we define and study its Nilsson series solutions associated to a weight
vector $w \in \RR_{\geq 0}^n$. We use the convention that
$\NN = \{0,1,2,\dots\}$.

We work in the Weyl algebra $D$ in $x_1,\dots, x_n$, $\del_1,\dots,\del_n$.
A left $D$-ideal $I$ is said to be \emph{holonomic} if
$\textrm{Ext}^i_D(D/I,D)=0$ for all $i \neq n$. In this case, the 
\emph{holonomic rank} of $I$, denoted by $\rank(I)$, which is by definition 
the dimension of the space of germs of holomorphic solutions of
$I$ near a generic nonsingular point, is finite, by a result of
Kashiwara (see Theorem~1.4.19 in~\cite{SST}).

\begin{definition}
\label{def:weight}
Let $w \in \RR^n$. The $(-w,w)$-\textbf{weight} of 
$x^u\del^v \in D$ is
$ -w \cdot u + w \cdot v$.
\end{definition}

This weight induces a partial order on the monomials in
the Weyl algebra.

\begin{definition}
\label{def:initial-term}
Let $w \in \RR^n$ and $f=\sum_{u,v}c_{u v}x^u\partial^v \in D$.
The \textbf{initial form} $\inww(f)$ of $f \neq 0$ with respect to $(-w,w)$
is the subsum of $f$ consisting of its (nonzero) terms of maximal $(-w,w)$-weight. 
If $I$ is a left $D$-ideal, its \textbf{initial ideal} with respect
to $(-w,w)$ is
\[
\inww(I) = \< \inww(f) \mid f \in I, f \neq 0 \> \subset D.
\] 
\end{definition}

\begin{remark}
\label{remark:initials in the polynomial ring}
We can restrict the ${(-w,w)}$-weight to the monomials in $\CC[\del]
\subset D$; the weight of $\del^v$ reduces to $w \cdot
v$. When $w \in \RR^n_{>0}$, the induced partial order on the monomials in
$\CC[\del]$ has $1$ as the unique smallest monomial.
If $J \subseteq \CC[\del]$ is an ideal, we define the
initial ideal  $\inw(J)$ as the ideal generated by the initial terms
of all non zero polynomials in $J$. 
\end{remark}

The following definition characterizes the weight vectors we
consider in this article. 

\begin{definition}
\label{convention:weight vector}
A vector $w\in \RR_{> 0}^n$ is a \textbf{weight vector} for $\HA$ if
there exists a strongly convex open rational polyhedral cone ${\cC}
\subset \RR^n_{> 0}$, 
with $w \in {\cC}$,
such that, for all $w'\in {\cC}$, we have
\[
\inw(I_A) \, = \, {\rm in}_{w'}(I_A) \quad \text{ and } \quad
\inww(H_A(\beta)) \, = \, {\rm in}_{(-w', w')} (H_A(\beta)).
\]
\end{definition}

Note that the cone $\cC$ is not unique and that
the assumptions on $w$ in Definition~\ref{convention:weight
  vector} imply that $\inw(I_A)$ is a monomial ideal.
It follows from the existence of the Gr\"obner fan (see~\cite{Gfan}
for the commutative version, and~\cite{Gfan-Walg} for the situation in
the Weyl algebra) that weight vectors form an open dense subset of
$\RR_{> 0}^n$. 
For an introduction to the theory of Gr\"obner bases in the Weyl
algebra, we refer to~\cite{SST}; this text is also
our main reference for
background on $D$-modules and hyper\-geo\-metric differential equations.

\begin{remark}
\label{remark:positive-weight-vectors}
Define the $A$-grading on $D$ via $\deg_A(x^u\del^v) = A\cdot(v-u)$. A left
$D$-ideal is $A$-homogeneous if it is generated by $A$-homogeneous
elements of $D$. Note that $\HA$ is $A$-homogeneous.

Suppose that the cone spanned by the columns of $A$ is strongly
convex, that is, there exists $h \in \RR^d$ such
that the vector $h \cdot A$ is coordinatewise positive.
Let $J$ be an $A$-homogeneous left $D$-ideal. 
Then for any $w \in \RR^n$ there exists $w' \in \RR^n_{>0}$ such that 
$\inww(J) = {\rm in}_{(-w',w')}(J)$. 
In fact, given $w \in \RR^n$, it is enough to choose a positive number $\lambda$ such that 
$w' = w + \, \lambda h\cdot A$ is coordinatewise positive. To see
this, note that if $f = \sum c_{uv}
x^u\del^v \in D$ is $A$-homogeneous, the vector
$A\cdot (v-u)$ is independent of $(u,v)$ for all 
$c_{uv} \neq 0$.
As
\[
- w'\cdot u + w' \cdot v = -w \cdot u + w \cdot v + 
\lambda \,[ \,h\cdot A \cdot (v-u) \,], 
\]
using $w'$ instead of $w$ simply adds a
constant to the weights of the terms in $f$, and
then  $\inww(f)={\rm{in}}_{(-w',w')}(f)$.
Thus, when $A$ is strongly convex, we can drop the
positivity assumption in Definition~\ref{convention:weight vector}.

An important special case in which the cone 
spanned by the columns of $A$ is strongly
convex is when
$(1, \dots, 1)$ belongs to
the $\QQ$-rowspan of $A$. This happens if and only if the
toric ideal $I_A$ is homogeneous with respect to the usual
$\ZZ$-grading of the polynomial ring $\CC[\del]$ given by
$\deg(\del_1)=\dots=\deg(\del_n)=1$. 
We use the convention that the word \emph{homogeneous} without
qualifications refers to the usual $\ZZ$-grading of $\CC[\del]$.
\end{remark}

If $w$ is a weight vector for $\HA$, and $\cC$ is an open cone as in
Definition~\ref{convention:weight vector}, 
the \emph{dual cone} $\cC^*$
consisting of elements $u \in \RR^n$ such that $u \cdot w' \geq 0$ for
all $w' \in \cC$ is strongly convex. Moreover, for any nonzero 
$u \in \cC^*$ and any $w'\in \cC$, we have $ u\cdot w' > 0$.

We are now ready to 
define the space of formal 
Nilsson series solutions of $\HA$ associated to a weight
vector $w$. Note that the name \emph{Nilsson class} is usually
reserved for multivalued functions that satisfy tempered growth
conditions (see Chapter~6.4 in~\cite{bjork}).
In this article, we work
with \emph{formal} Nilsson series, except when otherwise noted.

\begin{definition}
\label{defformal}
Let $w$ be a weight vector for the hyper\-geo\-metric system $\HA$.
Denote $\log(x) = (\log(x_1),\dots,\log(x_n))$.
A formal solution $\phi$ of $\HA$ is called a \textbf{basic
Nilsson solution of $\HA$ in the direction of
  $w$} if it has the form 
\begin{equation}
\label{eqn:Nilsson series}
\phi\,  
= \sum_{u \in C} x^{v+u} p_u(\log(x)), 
\end{equation}
for some vector $v\in \CC^n$, and it satisfies
\begin{enumerate}[\qquad 1.]
\item $C$ is contained in 
  ${\cC}^* \cap \ker_\ZZ(A)$, 
  where $\cC$ is an open
  cone containing $w$ as in Definition~\ref{convention:weight vector},
\label{defformal1}
\item The $p_u$ are polynomials, and there exists $K\in \ZZ$ such that
  $\deg(p_u) \leq K$ for all 
  $u \in C$,
\label{defformal2}
\item $p_0 \not=0$. \label{defformal3}
\end{enumerate}
The set 
$\supp(\phi) = \{u \in C \mid p_u \neq 0\} \subset \ker_{\ZZ}(A)$ 
is called the
\textbf{support} of $\phi$.

The $\CC$-span of the basic Nilsson solutions of $\HA$ in the direction of
$w$ is called the \textbf{space of formal Nilsson series solutions of
  $\HA$ in the direction of $w$} and is denoted by $\Fw(\HA)$.
\end{definition}

 In what
follows, we make a detailed study of $\Fw(\HA)$. One of our main
results, Theorem~\ref{ultimoteorema}, states that, for fixed $A$ and $\beta$ there
exists a weight vector $w$ such that $\dim(\Fw(\HA)) = \rank(\HA)$.

Let $w$ be a weight vector for $\HA$.
If $\phi = \sum_{a,b} c_{ab} x^a \log(x)^b$ is a non zero 
element of $\Fw(\HA)$ then, by Proposition~2.5.2 in~\cite{SST}, the set of real parts
$
\{ 
\textrm{Re}(a \cdot w)
    \mid  c_{ab} \neq 0 \text{ for some } b
\}
$
achieves a (finite) minimum, denoted by $\mu(\phi)$,
and the subseries of $\phi$ whose
terms are $c_{ab} x^{a} \log(x)^b$ such that $c_{ab} \neq 0$ and
$\textrm{Re}(a \cdot w)  = \mu(\phi)$ is finite.
We call this finite sum the
\emph{initial series of $\phi$ with respect to $w$} and we denote it by
$\inw(\phi)$. 
The reason for including the third condition in
Definition~\ref{defformal} is to ensure that $\inw(\phi) = x^vp_0(\log(x))$.
Non-basic Nilsson solutions of $\HA$ will play a role in
Section~\ref{sec:logarithm free}.

\begin{remark} 
\label{rem:simplebasic}
Given a weight vector $w$, we may replace the first requirement in
Definition~\ref{defformal} by either of the following equivalent conditions:
\begin{enumerate}[ i)] 
\item  \label{rem:1} 
$C \subset \ker_{\ZZ}(A)$ and 
there exists an open neighborhood $U$ of $w$ such that, for all $w'\in
U$ and all $u \in C \minus \{ 0 \}$, we have
$ w'\cdot u > 0$.
\item \label{rem:2} 
There exist $\RR$-linearly independent
$\gamma_1, \dots, \gamma_n \in \QQ^n$ with 
$w\cdot \gamma_i  > 0$ for all $i=1, \dots, n$, such that 
$C \subset \big( \RR_{\geq 0} \gamma_1 + \dots + \RR_{\geq 0} \gamma_n
\big) \cap \ker_{\ZZ}(A)$.
\end{enumerate}

To see that~\eqref{rem:1} is equivalent to the first condition 
in Definition~\ref{defformal}, note first that, if the latter is true, we
may use $U = \cC$. 
The proof that~\eqref{rem:1} and~\eqref{rem:2} are equivalent is straightforward.
If~\eqref{rem:2} holds, we may take $\cC$ equal to the interior of 
the cone $\RR_{\geq 0} \gamma_1 + \dots + \RR_{\geq 0} \gamma_n$.
\end{remark}

The following lemma allows us to manipulate formal Nilsson solutions of
$\HA$ in the direction of a weight vector.

\begin{lemma}
\label{thm:aux-ini-indep}
Let $\phi_1,\dots,\phi_k \in \Fw(\HA)$.
\begin{enumerate}[1.]
\item If the initial series
$\inw(\phi_1), \dots,\inw(\phi_k)$ are $\CC$-linearly independent,
then so are the series
$\phi_1,\dots,\phi_k$. \label{t-item2}
\item If $\phi_1,\dots,\phi_k$ are $\CC$-linearly independent,
there exists a $k \times k$ complex matrix $(\lambda_{ij})$ such that
the initial series of 
$\psi_i = \sum_{j=1}^k \lambda_{ij} \phi_j$ for $i=1,\dots,k$
are $\CC$-linearly independent.
\label{t-item3}
\end{enumerate}
\end{lemma}

\begin{proof}
Combine Theorem~2.5.5, Lemma~2.5.6(2) 
and Proposition~2.5.7 from~\cite{SST},
which hold for formal Nilsson series.
\end{proof}

\begin{definition} 
\label{exponent}
A vector $v \in \CC $ is an \textbf{exponent} of $\HA$ with respect to a weight
vector $w$ if $x^v$ is a solution of $\inww(\HA)$. 
\end{definition}

Note that if $w$ is a weight vector for $\HA$ and $\cC$ is a strongly convex
open cone as in Definition~\ref{convention:weight vector}, then for any
$w'\in \cC$, the exponents of $\HA$ with respect to $w$ and $w'$
coincide, because $\inww(\HA) = {\rm in}_{(-w',w')}(\HA)$. Moreover,
the basic Nilsson solutions of 
$\HA$ in the direction of $w$ and $w'$ are the same, and therefore
$\Fw(\HA) = \cN_{w'}(\HA)$. 

Let $w$ be a weight vector for $\HA$.
Since $\HA$ is a holonomic ideal for all 
$\beta$ (see Theorem~3.9 in~\cite{Ado}), the ideal $\inww(\HA)$
is holonomic as well, and its holonomic rank is at most 
$\rank(\HA)$, by Theorem~2.1 in~\cite{SST}. This implies that the
set of exponents of $\HA$ with respect to $w$ is finite, since
the monomial functions corresponding to different exponents are
linearly independent.

\begin{lemma} 
\label{rem:initial}
If $\phi $ is a basic Nilsson solution of $\HA$ in the direction of $w$
as in \eqref{eqn:Nilsson series},
then $v$ is an exponent of $\HA$ with respect to $w$.
\end{lemma}

\begin{proof}
Since $w$ is a weight vector, $w \cdot u > 0$ for all nonzero $u \in
\cC^*$. As $p_0 \neq 0$, we have $\inw(\phi) = x^vp_0(\log(x))$, and therefore
$x^v p_0(\log(x))$ is a solution of $\inww(\HA)$ 
by Theorem~2.5.5 in~\cite{SST}. But then $x^v$ is a solution of 
$\inww(\HA)$ by Theorems~2.3.3(2) and~2.3.11 in~\cite{SST}. 
\end{proof}

We compare the dimension of the space of formal
Nilsson solutions of $\HA$ in the direction of $w$ with the holonomic rank of the
associated initial ideal.

\begin{proposition}
\label{dimformmenorrankini}
Let $w$ be a weight vector for $\HA$. 
Then
 \begin{equation}
\dim_{\CC}(\Fw(\HA)) \leq \rank(\inww(\HA)).
 \end{equation}
\end{proposition}

\begin{proof}
Choose $\psi_1,\dots,\psi_k$ linearly independent elements of $\Fw(\HA)$. 
The second part of
Lemma~\ref{thm:aux-ini-indep} allows us to assume that
$\textrm{in}_w(\psi_1),\dots,$ $\textrm{in}_w(\psi_k)$ are linearly
independent solutions of the initial system $\inww(\HA)$.  These
initial series have a non empty common open domain of convergence since
they have a finite number of terms. Therefore
$\dim_{\CC}(\Fw(\HA))$ cannot exceed the holonomic rank of $\inww(\HA)$.
\end{proof}

We show in Corollary~\ref{coro:dimequalrankini} that this inequality
is, in fact, an equality for generic $\beta$.
If $I_A$  is a homogeneous ideal, then Theorems~2.4.9,~2.5.1,
and~2.5.16 in~\cite{SST} imply that
$\dim_\CC \Fw(\HA)$, $\rank(\inww(\HA))$, and $\rank(\HA)$ are the same.
However, if $I_A$ is not homogeneous, 
$\rank(\inww(\HA))$ does not always equal $\rank(\HA)$ (see 
Corollary~\ref{coro:notalways}).

\section{Homogenization of formal Nilsson solutions of $\HA$}
\label{sec:invert-der}

The goal of this section is to obtain the solutions of the system
$\HA$ by solving a related hyper\-geo\-metric system that is regular holonomic.  
For generic parameters, this idea was used in other works, such 
as~\cite{OT}; here, we require no genericity hypotheses on $\beta$. The
key concept is that of homogenization.

\begin{notation}
\label{not:homogenization}
Throughout this article, the letter $\rho$ is used to
indicate the homogenization of various objects:
polynomials, ideals, and later on, Nilsson
series.

If $f \in \CC[\del_1,\dots,\del_n]$ is a polynomial,
we denote by $\rho(f) \in \CC[\del_0,\del_1,\dots,\del_n]$
its homogenization, that is,
\[
f = \sum_{u \in \NN^n} c_u \del^u  \Longrightarrow 
\rho(f) = \sum_{u\in \NN^n} c_u \del_0^{\deg(f)-|u|}\del^u, \quad |u| = u_1 +
\dots + u_n.
\]
If $I \subseteq \CC[\del_1,\dots,\del_n]$ is an ideal, then
$\rho(I)\subseteq \CC[\del_0,\del_1,\dots,\del_n] $ denotes the ideal generated by the homogenizations 
$\rho(f)$ for all $f \in I$.

If $A=[a_{ij}]$ is a $d\times n$ integer matrix, then
$\rho(A) \in \ZZ^{(d+1)\times(n+1)}$ is obtained by attaching a column of
zeros to the left of $A$, and then attaching 
a row of ones to the resulting matrix, namely
\[
\rho(A)=\left[\begin{array}{cccc}
 1 & 1 & \ldots & 1 \\
0 & a_{11} & \ldots & a_{1n} \\
\vdots & \vdots & & \vdots \\
0 & a_{d1} & \ldots & a_{dn}
\end{array}\right].
\]
Note that $\rho(I_A) = I_{\rho(A)}$.
\end{notation}

Let $w$ be a weight vector for $\HA$ and let $\cC$ be a cone as in
Definition~\ref{convention:weight vector}. In particular $w\in
\RR^n_{>0}$. For fixed $\beta_0 \in \CC$, consider
the (regular holonomic) hyper\-geo\-metric system associated to the matrix
$\rho(A)$ from Notation~\ref{not:homogenization} and the vector 
$(\beta_0,\beta) \in\CC^{n+1}$, which is denoted by $\HrA$.

\begin{remark}\label{rmk:0w}
Since $(1,\dots,1)$ belongs to the rational rowspan of $\rho(A)$,
weight vectors for $\HrA$ are not required to have positive
coordinates (see Remark~\ref{remark:positive-weight-vectors}). We wish
to use $(0,w)$ as a weight vector for $\HrA$, but 
this vector may not satisfy the definition. This can be remedied by
perturbing $w$ as follows.
The set of weight vectors for $\HrA$ is an open dense subset of
$\RR^{n+1}$. Therefore, given $w$ (and $\cC$) as before, there exists $\alpha=
(\alpha_0,\dots,\alpha_n) \in \RR^{n+1}$ and $\varepsilon_0>0$ 
such that $(0,w)+ \varepsilon \alpha$ is a weight vector for $\HrA$
for all $0<\varepsilon <\varepsilon_0$. But then the argument in
Remark~\ref{remark:positive-weight-vectors} shows that
$(0,w) + \varepsilon \alpha -\varepsilon \alpha_0 (1,\dots,1)$ is also
a weight vector for $\HrA$.
If $\varepsilon$ is sufficiently small, then 
$w'=w + \varepsilon((\alpha_1,\dots,\alpha_n)-\alpha_0(1,\dots,1))$
belongs to an open cone $\cC$ as in Definition~\ref{convention:weight vector}. This means
that we can use $w'$ instead of $w$ as weight vector for $\HA$, with the same
open cone, initial ideals, and basic Nilsson solutions as $w$, and
guarantee that $(0,w')$ is a weight vector for $\HrA$. 
The previous argument justifies assuming, as { we do from now on}, that
any time we choose a weight 
vector $w$ for $\HA$, the vector $(0,w)$ is a \emph{weight vector} for $\HrA$.
\end{remark}

We choose a weight vector $w$, and we wish to use the auxiliary 
regular
system $\HrA$ to study the solutions of $\HA$. The matrix
$\rho(A)$ is fixed, but we have freedom in the choice of the parameter
$\beta_0 \in \CC$, and it is convenient to assume that $\beta_0$ is
generic. The correct notion of genericity for $\beta_0$ can be found
in Definition~\ref{defi:homog-value}. Under that hypothesis, our
objective is to construct
an injective linear map
\begin{equation}
\label{eq:homog-sol}
\rho : \Fw(\HA) \longrightarrow
\mathscr{N}_{(0,w)}(H_{\rho(A)}(\beta_0,\beta)),
\end{equation}
whose image is described in Theorem~\ref{thm:image of
  rho}. Since by definition $\Fw(\HA)$ has a basis consisting of
basic Nilsson solutions of $\HA$ in the direction of $w$,
it is enough to define our map on those
series, and check that their images are linearly independent. 
For some weight vectors, $\rho$ is guaranteed to be
surjective (Proposition~\ref{cor:surj-case}). 
However, if the cone over the columns of $A$ is strongly convex and
$I_A$ is not homogeneous, there
always exist weights for which surjectivity 
fails (Proposition~\ref{coro:non-surj}).

Let $\phi=\sum_{u\in C} x^{v+u} p_u(\log(x))$ be a basic Nilsson
solution of $\HA$ in the direction of $w$ as 
in~\eqref{eqn:Nilsson series}.
Since $\phi$ is annihilated by the Euler operators
$E_1-\beta_1,\dots,E_d-\beta_d$, the polynomials $p_u$
appearing in $\phi$ belong to the symmetric algebra of the lattice
$\ker_{\ZZ}(A)$ by Proposition~5.2 in~\cite{S}. 

If $v$ is a vector, denote by $|v|$ the sum of its coordinates. Then,
for any $\gamma \in \ker_{\ZZ}(A)$, $(-|\gamma|,\gamma) \in
\ker_{\ZZ}(\rho(A))$. This inclusion
$\ker_{\ZZ}(A) \hookrightarrow \ker_{\ZZ} (\rho(A))$ 
induces an injection (denoted by $\; \widehat{\cdot}\;$) between the corresponding
symmetric algebras of the lattices of $\ker_{\ZZ}(A)$ and $\ker_{\ZZ}(\rho(A))$.
In concrete terms, let $\{ {\gamma}_1,\dots,{\gamma}_{n-d} \} \subset \ZZ^n $ 
be a $\ZZ$-basis of $\ker_{\ZZ}(A)$. We can write an element $p$
of the symmetric algebra of $\ker_{\ZZ}(A)$ as follows:
\[
p(t_1,\dots,t_n) = 
\sum_{\alpha \in \NN^{n-d}} c_{\alpha} \prod_{j=1}^{n-d} 
({\gamma}_j\cdot(t_1,\dots,t_n))^{\alpha_j},
\]
and then we have
\begin{equation}
\label{eq:qu}
\hat{p}(t_0,t_1,\dots,t_n) = 
\sum_{\alpha \in \NN^{n-d}} c_{\alpha} 
\prod_{j=1}^{n-d} ((-|{\gamma}_j|)\, t_0 +{\gamma}_j\cdot(t_1,\dots,t_n))^{\alpha_j}.
\end{equation}
Note that $\hat{p}(\log(x_0),\dots,\log(x_n))$ specializes to
$p(\log(x))$ when $x_0=1$, or equivalently, when $\log(x_0)=0$.

The formal definition of the homogenization of a basic Nilsson
solution 
\[\phi = \sum_{u\in C} x^{v+u} p_u(\log(x))\] 
of $\HA$ in the direction of $w$ is:
\begin{equation}
\label{eqn:definition of rho}
\rho(\phi) = \sum_{u\in C} \del_0^{|u|} x_0^{\beta_0-|v|}
x^{v+u}\widehat{p_u}(\log(x_0),\dots,\log(x_n)). 
\end{equation}
If $|u|\geq 0$ for all $u \in C$, the
above formula makes sense, and it easily checked that $\rho(\phi)$
is a basic Nilsson solution of $\HrA$ in the direction of $(0,w)$. 
The bulk of the work in this section concerns the definition and
properties of the operator $\del_0^{k}$ when $k \in \ZZ_{<0}$.

We point out that there is one case when the elements of the supports
of all basic Nilsson solutions of $\HA$ in the direction of $w$ are
guaranteed to have non negative coordinate sum, namely
when the weight vector $w$ is close to
$(1,\dots,1)$. We make this notion precise in the following
definition. 

\begin{definition}
\label{def:perturbation}
Let $w$ be weight vector for $\HA$. We say that $w$ is a
\textbf{perturbation} of  
$w_0 \in \RR_{>0}^n$ if there exists an open cone $\cC$ as in
Definition~\ref{convention:weight vector} with $w \in \cC$, such that
$w_0$ lies in the closure of $\cC$.
\end{definition}

Suppose that $\phi = x^v \sum_{u \in C} x^up_u(\log(x))$ 
is a basic Nilsson solution
of $\HA$ in the direction of a weight vector $w$ which is a perturbation
of $(1, \dots,1)$. Since $u \in C$ implies 
$u \in \cC^* = (\overline{\cC})^*$, we have $u \cdot w \geq 0$ for all
$u \in C$. But then, as $w$ is a perturbation of
$(1,\dots,1)$, it follows that $|u|= u \cdot (1,\dots,1) \geq 0$ for
all $u \in C$. Also, $C \subseteq \ZZ^n$ implies $|u| \in \ZZ$ for $u
\in C$. Therefore, the operator $\partial_0^{|u|}$ is
defined, and so is~\eqref{eqn:definition of rho}.

As we mentioned before, in order to work with other weight vectors, we
must define the operator $\del_0^k$ when $k$ is negative. 
We need to introduce key combinatorial objects from~\cite{STV}, the
 standard pairs of a monomial ideal $I$, which correspond to
to the associated primes of $I$ (see~\cite{STV}).

\begin{definition}
\label{def:stdpairs}
Let $I \subseteq \CC[\del_1,\dots,\del_n]$ be a monomial ideal. Consider the variety
\[
\VV=
\overline{
\{ u \in \NN^n \subset \CC^n  \mid  x^u \not \in I \}
}^{\rm{\scriptstyle{Zariski}}} 
\]
whose irreducible components are integer translates of coordinate spaces. For
$\sigma \subseteq \{1,\dots,n\}$, denote
$\CC^{\sigma} = \{ (t_1,\dots,t_n) \mid t_i = 0 \;\mbox{for}\; i \not \in \sigma \}$.
A pair $(\del^{\alpha},\sigma)$, where $\alpha \in \NN^n$ and 
$\sigma \subseteq \{1,\dots,n\}$, 
is called a \textbf{standard pair of $I$} if 
$\alpha+\CC^{\sigma}$ is an irreducible component of $\VV$ and
$\alpha$ is the coordinatewise minimal element of 
$(\alpha+\CC^{\sigma})\cap\NN^n$.
\end{definition}

Note that, if $(\del^{\alpha},\sigma)$ is a standard pair of a
monomial ideal $I\subseteq \CC[\del]$, then $\alpha_i = 0$ for all $i
\not \in \sigma$. 
A standard pair is \emph{top-dimensional} if the 
corresponding associated prime is minimal, and it is \emph{embedded}
otherwise. We denote by $\cT(I)$ the set of top-dimensional
standard pairs of $I$, and define
$\tp(I)$ to be the intersection
\[
\tp(I)=\bigcap_{(\partial^{\alpha}, \, \sigma)\in\cT(I)}
          \langle\partial_i^{\alpha_i+1} \mid i\notin\sigma\rangle.
\]

\begin{definition}
\label{def:fake-exps}
A vector $v \in \CC^n$ is a \textbf{fake exponent} of $\HA$ with
respect to $w$ if the function $x^v$ is a solution of the left
$D$-ideal $\inw(I_A)+\<E-\beta\>$.
\end{definition}

Note that, as $\inw(I_A)+\<E-\beta\> \subseteq \inww(\HA)$, all
exponents of $\HA$ with respect to $w$ (Definition~\ref{exponent}) are
also fake exponents. In general, this containment is strict, but fake
exponents are easier to study combinatorially than (true) exponents.
This is illustrated by the following result, a consequence
of Lemma~4.1.3 in~\cite{SST}, which links fake 
exponents and standard pairs.

\begin{lemma}
\label{lemma:fake-exps}
Let $w$ be a weight vector for $\HA$ and let $(\del^{\alpha},\sigma)$ be a
standard pair of $\inw(I_A)$. If there exists $v \in \CC^n$ such that
$A\cdot v= \beta$ and $v_i = \alpha_i \,$ for all $i \not \in \sigma$,
then $v$ is a fake exponent of $\HA$ with respect to $w$. All fake
exponents arise this way.
\end{lemma}

We also need to consider standard pairs of initial ideals of $I_{\rho(A)}$.
Recall that the initial ideal $\inow(I_{\rho(A)})$ is a monomial ideal, as $(0,w)$
is a weight vector for $\HrA$.

\begin{definition}
\label{def:pass through zero}
A standard pair of
$\inow(I_{\rho(A)})$ is said to \textbf{pass through zero} if $0 \in \sigma$.
\end{definition}

\begin{remark}
\label{remark:exponent notation}
If a fake exponent $v$ of $\HA$ exists for a 
standard pair $(\del^{\alpha},\sigma)$ of $\inw(I_A)$, it is
unique by Corollary~3.2.9 in~\cite{SST}; this justifies 
using the notation $v=\beta^{(\del^{\alpha},\sigma)}$.
\end{remark}

\begin{definition}
\label{def:arises}
We say that a basic Nilsson solution of $\HA$ as 
in~\eqref{eqn:Nilsson series}
is \textbf{associated} to the standard pair $(\del^{\alpha},\sigma)$
if $v$ is the (fake) 
exponent corresponding to this standard pair. 
\end{definition}

The following definition gives the correct notion of genericity for
$\beta_0 \in \CC$ so that we can study the solutions of $\HA$ using
those of $\HrA$.

\begin{definition}
\label{defi:homog-value}
Let $w$ be a weight vector for $\HA$.
We say that $\beta_0 \in \CC$ is
a \textbf{homogenizing value} for $A$, $\beta$, and $w$ if 
$\beta_0 \notin \ZZ$ and
for any fake exponent $v$ of $\HA$ with respect to $w$, the quantity
$v_0$ defined by 
\[v_0: = \beta_0-\sum_{j=1}^n v_j\]
is not an integer number.
\end{definition}

Given a weight vector $w$ for $\HA$, we fix a homogenizing value
$\beta_0$ for $A$, $\beta$, and $w$. Let $\phi$ as
in~\eqref{eqn:Nilsson series} be a basic
Nilsson solution of $\HA$ in the direction of $w$. We want to construct a basic
Nilsson solution $\rho(\phi)$ of $\HrA$ in the direction of $(0,w)$.
The following lemma tells us how to differentiate
logarithmic terms.

\begin{lemma}(Lemma~5.3 in~\cite{S})
\label{lem:form-of-der}
Let $h$ be a polynomial in $n$ variables, $\nu \in \NN^n$ and $s \in \CC^n$.
Then
\[
\del^\nu x^{s} h(\log(x)) = x^{s-\nu} \left( \sum _{0 \leq \nu' \leq \nu}
\lambda_{\nu'}  \bigg[\del^{\nu-\nu'} h\bigg](\log(x))\right),
\]
where the sum is over nonnegative integer vectors $\nu'$ that are
coordinatewise smaller than $\nu$, and the
$\lambda_{\nu'}$ are certain complex numbers. \qed
\end{lemma}

The following result allows us to define $\del_0^k$ when $k$ is a
negative integer.

\begin{lemma}
\label{lema:invert-del0}
Let $\hat{p}$ be a polynomial in $n+1$ variables, and $s \in \CC^n$. 
If $s_0\in\CC$ and $s_0 \not= -1$, there
exists a unique polynomial $\hat{q}$ with 
$\deg(\hat{q}) = \deg(\hat{p})$   
such that
\begin{equation}
\label{eqn:antider}
\del_0 \big[ x_0^{s_0+1} x^s \hat{q}(\log(x_0),\dots,\log(x_n)) \big] =
x_0^{s_0}x^s \hat{p}(\log(x_0),\dots,\log(x_n)).
\end{equation}
\end{lemma}

\begin{proof}
Writing 
\[
\hat{p}(\log(x)) = \sum_{i=0}^k p_i(\log(x_1),\dots,\log(x_n)) \log(x_0)^i
\]
and
\[
\hat{q}(\log(x)) = \sum_{i=0}^k q_i(\log(x_1),\dots,\log(x_n)) \log(x_0)^i,
\] 
we can 
equate coefficients in~(\ref{eqn:antider}) to obtain
\[
p_k = (s_0+1)q_k \; ; \quad p_i = (s_0 +1) q_i + (i+1) q_{i+1} \quad 0
\leq i \leq k-1.
\]
Therefore
\[
q_{i} = \sum_{\ell = 0}^{k-i} (-1)^\ell \frac{\prod_{j=1}^{\ell}
(i+j)}{(s_0+1)^{\ell}} p_{i+\ell} 
\quad 0 \leq \ell \leq k,
\]
where the empty product is defined to be $1$.
\end{proof}

\begin{definition} 
\label{not:delta0-1}
With the notation of Lemma~\ref{lema:invert-del0}, define
\[
\del_0^{-1} \big[ x_0^{s_0}x^s \hat{p}(\log(x_0),\dots,\log(x_n)) \big] = 
x_0^{s_0+1} x^s \hat{q}(\log(x_0),\dots,\log(x_n)).\]
Note that if $s_0\not=-2, \dots, -k$, 
the construction
of  $\del_0^{-1}$ can be iterated $(k-1)$ times. We denote by
$\del_0^{-k} \big[ x_0^{s_0}x^s \hat{p}(\log(x_0),\dots,\log(x_n)) \big]$
the outcome of this procedure.
\end{definition}

\begin{lemma}
\label{lema:still-homog}
Use the same notation and hypotheses as in
Lemma~\ref{lema:invert-del0}, and assume furthermore that
$x_0^{s_0} x^s \hat{p}(\log(x))$ is a solution of $\< E_0 - \beta_0,
E-\beta\>$. Then
$\del_0^{-1}[x_0^{s_0} x^s \hat{p}(\log(x))]$ is a solution of
$\< E_0 -(\beta_0+1), E-\beta\>$.
If $s_0 \neq -2,\dots, -k$, then 
$\del_0^{-k}[x_0^{s_0} x^s \hat{p}(\log(x))]$ is a solution of
the system
$\< E_0 -(\beta_0+k), E-\beta\>$.
\end{lemma}

\begin{proof}
If $i > 0$, $\del_0 (E_i - \beta_i) = (E_i - \beta_i) \del_0$, so that
\begin{align*}
\del_0 (E_i -\beta_i) (\del_0^{-1}x_0^{s_0} x^s \hat{p}(\log(x))) & = 
(E_i -\beta_i) \del_0 (\del_0^{-1}x_0^{s_0} x^s \hat{p}(\log(x))) \\
& =  (E_i-\beta_i) x_0^{s_0} x^s \hat{p}(\log(x))  \\
& = 0.
\end{align*}
This means that $(E_i-\beta_i) (\del_0^{-1}x_0^{s_0} x^s \hat{p} (\log(x)))$ is
constant 
with respect to $x_0$. On the other hand, $s_0+1 \neq 0$ and 
$(E_i-\beta_i) (\del_0^{-1}x_0^{s_0} x^s  \hat{p}(\log(x)))$
is a multiple of $x^{s_0+1}$. Thus, in order
to be constant with respect to $x_0$,  $(E_i-\beta_i)
(\del_0^{-1}x_0^{s_0} x^s \hat{p} (\log(x)))$ has to vanish. 
For $i=0$, the argument is similar since
$\del_0 (E_0 -(\beta_0+1))$ $= (E_0 -\beta_0)\del_0$.
The last assertion follows by induction on $k$.
\end{proof}

\begin{lemma}
\label{lema:commute}
Let $\nu \in \NN^n$, $k \in \ZZ$, and assume that $s_0 \notin \ZZ$. Then
\[
\del^\nu \big[ \del_0^k \big[ x_0^{s_0}x^s \hat{p}(\log(x_0),\dots,\log(x_n)) \big]
\big] =
\del_0^k \big[ \del^\nu \big[  x_0^{s_0}x^s \hat{p}(\log(x_0),\dots,\log(x_n)) \big]
\big] .
\]
\end{lemma}

\begin{proof}
This is clear if $k \geq 0$. For $k<0$, the result follows by
induction from the
uniqueness part of Lemma~\ref{lema:invert-del0}.
\end{proof}

We are now ready to define the homogenization of a basic Nilsson
solution of $\HA$.

\begin{definition}
\label{def:homog-sol}
Let 
$\phi = x^v \sum_{u \in C} x^u p_u(\log(x))$ be a basic Nilsson
solution of $\HA$ 
in the direction of a weight vector $w$, and let $\beta_0\in\CC$ a
homogenizing value for $A$, 
$\beta$ and $w$, so that
$v_0=\beta_0-\sum_{i=1}^n v_i$ is not an integer. We
define
\[
\rho(\phi) = \sum_{u \in C} \del_0^{|u|} 
\left[ x_0^{v_0} x^{v+u}
\widehat{p_u}(\log(x_0),\log(x_1),\dots,\log(x_n) )\right] ,
\]
where $\widehat{p_u}$ is obtained from $p_u$ as in~(\ref{eq:qu}).
\end{definition}

Note that if $\phi_1$ and $\phi_2$ are basic Nilsson solutions of
$\HA$ and $\lambda \in \CC$ is such that $\phi_1+\lambda \phi_2$ is also basic, then
$\rho(\phi_1+\lambda \phi_2) = \rho(\phi_1)+\lambda\rho(\phi_2)$, as $\del_0^k$ is linear. 
We now verify that the homogenization of a basic Nilsson solution of
$\HA$ is a basic Nilsson solution of $\HrA$.

\begin{proposition}\label{propo:Nilsson-to-Nilsson}
Let $w$ be a weight vector for $\HA$ and let $\beta_0$ be a fixed
homogenizing value for $A$, $\beta$ and $w$. 
For any   
$\phi = x^v \sum x^u p_u(\log(x))$ basic Nilsson solution of $\HA$
in the direction of $w$, the (formal) series $\rho(\phi)$ from
Definition~\ref{def:homog-sol} is a basic Nilsson solution of $\HrA$
in the direction of $(0,w)$.
We extend $\rho$ linearly to obtain a map
\[
\rho: \Fw(\HA) \to \mathscr{N}_{(0,w)}(H_{\rho(A)}(\beta_0,\beta)).
\]
\end{proposition}

\begin{proof}
Recall our assumption that, if $w$ is a weight vector for $\HA$,
$(0,w)$ is a weight vector for $\HrA$.

We first show that the series $\rho(\phi)$ has the shape required in
Definition~\ref{defformal}. Conditions~\ref{defformal2} and~\ref{defformal3}
are clearly satisfied by the construction of the polynomials $\hat{p}_u$ and
Lemma~\ref{lem:form-of-der}. Thus, it is enough
to verify that $\rho(\phi)$ satisfies
condition~(\ref{rem:1}) from Remark~\ref{rem:simplebasic}. The support of
$\phi$ is in bijection with the support of $\rho(\phi)$ via $u \mapsto
(-|u|, u)$, which sends $\ker_\ZZ(A)$ into $\ker_\ZZ(\rho(A))$. We can assume
that $C$ is the intersection of $\ker_\ZZ(A)$ with the dual $\cC^*$ of an open
cone $\cC$ such that  its closure
$\overline{\cC}$ is a strongly convex rational polyhedral cone of maximal dimension.
Let $\{\gamma_1, \dots, \gamma_m\}$ be a Hilbert basis of 
$\cC^* \cap \ker_\ZZ(A)= (\overline{\cC})^* \cap \ker_\ZZ(A)$. 
Then $ w' \cdot \gamma_i > 0$ for all $w' \in \cC$ and all 
$i = 1, \dots, m$. Let $\delta > 0$
such that, for all $\be \in \RR_{>0}^{n+1}$ whose Euclidean distance
to the origin is $||\be|| < \delta$, and all $i=1,\dots, m$, we have  
$\big[ (0,w) + \be \big] \cdot (-|\gamma_i|, \gamma_i) > 0$.
It follows that for
any non zero $u \in C$ and any $\tilde{w}$ in the ball centered at
$(0,w)$ with radius $\delta$, we have 
$\tilde{w} \cdot (-|u|, u) > 0$,
which proves our claim.

Now we prove that $\rho(\phi)$ is a formal solution of $\HrA$.
Since $\widehat{p_u}$ belongs to the symmetric algebra of
$\ker_{\ZZ}(\rho(A))$, and 
\[
\rho(A) \cdot (v_0,v+u) =
(\beta_0-|v|+|v|+|u|, A \cdot (v+u)) = (\beta_0+|u|,\beta),
\] 
the term
$x_0^{v_0} x^{v+u} \widehat{p_u}(\log(x_0),\log(x_1),\dots,\log(x_n))$
is a solution of the system of Euler operators 
$\<E_0-(\beta_0+|u|),E_1-\beta_1,\dots, E_d-\beta_d\>$.
By Lemma~\ref{lema:still-homog} with $(s_0,s)=(v_0,v+u)$ and $k=-|u|$,
each term of $\rho(\phi)$ is therefore a solution of 
$\< E-(\beta_0+|u|-|u|), E_1-\beta_1,\dots,E_d-\beta_d \>$.

To verify that the elements of $I_{\rho(A)}$ annihilate $\rho(\phi)$,
first note that 
Lemma~\ref{lem:form-of-der} implies that, for any $\mu \in \ker_{\ZZ}(A)$,
\[
\del^{\mu_+} x^{v+u} p_{u}(\log(x)) = \del^{\mu_-}x^{v+u-\mu}p_{u-\mu}(\log(x)),
\]
because $\phi$ is a solution of $\HA$.

We claim that
\[
\del^{\mu_+} x_0^{v_0}x^{v+u} \widehat{p_{u}}(\log(x)) 
= \del^{\mu_-}x_0^{v_0}x^{v+u-\mu}\widehat{p_{u-\mu}}(\log(x)).
\]
To see this, use Lemma~\ref{lem:form-of-der} and the fact that,
if $i>0$ and
$p$ is an element of the symmetric algebra of $\ker_{\ZZ}(A)$, then
$\del_ip$ is also in the symmetric algebra of $\ker_{\ZZ}(A)$, and
$\del_i \widehat{p} = \widehat{\del_i p}$.

Now, using Lemma~{\ref{lema:commute}} with $(s_0,s)=(v_0,v+u)$ and
$k=|u|$, and the fact that 
$v_0 =\beta_0-|v| \notin \ZZ$,
we obtain from
\[
\del_0^{|u|}\del^{\mu_+} x_0^{v_0}x^{v+u} \widehat{p_{u}}(\log(x)) 
=
\del_0^{|\mu|}\del_0^{|u|-|\mu|}\del^{\mu_-}x_0^{v_0}x^{v+u-\mu}\widehat{
p_{u-\mu}}(\log(x)),
\]
that
\[
\del^{\mu_+} \left[ \del_0^{|u|} x_0^{v_0}x^{v+u}
\widehat{p_{u}}(\log(x)) \right] 
= \del_0^{|\mu|} \del^{\mu_-} 
\left[ \del_0^{|u|-|\mu|}x_0^{v_0}x^{v+u-\mu}
\widehat{p_{u-\mu}}(\log(x)) \right].
\]
Assuming  $|\mu|>0$, we conclude
\begin{align*}
\del^{(-|\mu|,\mu)_+}   \big[ \del_0^{|u|} x_0^{v_0}x^{v+u} &
\widehat{p_{u}}(\log(x)) \big] 
=  \\
& \del^{(-|\mu|,\mu)_-} 
\left[ \del_0^{|u|-|\mu|}x_0^{v_0}x^{v+u-\mu}\widehat{p_{u-\mu}}(\log(x))
\right].
\end{align*}
\end{proof}

The next result shows that $\rho$ is one to one. 
Therefore, the inverse map $\rho^{-1}$ allows us to obtain Nilsson solutions
of $\HA$ from Nilsson solutions of the regular holonomic system $\HrA$. 

\begin{theorem}
\label{thm:image of rho}
Let $w$ be a weight vector for $\HA$ and $\beta_0$ a homogenizing
value for $A$, $\beta$ and $w$.
The linear map 
\[
\rho:\Fw(\HA) \rightarrow {\mathscr{N}}_{(0,w)}(\HrA)
\]
is injective and its image is 
spanned by basic Nilsson solutions of $\HrA$ 
in the direction of $(0,w)$ associated
to standard pairs of $\inow(I_{\rho(A)})$ that pass through zero.
\end{theorem}

\begin{proof}
If $\phi=\sum_{u\in C} x^{v+u} p_u(\log(x))$ is a basic Nilsson
solution of $\HA$ in the direction of $w$, then we have
$\inow(\rho(\phi)) = x_0^{v_0} x^v \widehat{p_0}(\log(x_0),\dots,\log(x_n))$.
Choose a basis of $\Fw(\HA)$ consisting of basic Nilsson series
whose initial terms are linearly independent (use the second part of
Lemma~\ref{thm:aux-ini-indep}). Then the initial series of their images
are also linearly independent, as
$\hat{p}(1,\log(x)) = p(\log(x))$.
Now apply the first part of 
Lemma~\ref{thm:aux-ini-indep} to complete the proof that $\rho$ is injective.

Observe that, by construction, $\rho(\Fw(\HA))$ is contained in
the span of the basic Nilsson solutions of $\HrA$ corresponding to
standard pairs that pass through zero, because the powers of $x_0$
appearing in $\rho(\phi)$ are non integer for any basic Nilsson
solution $\phi$ of $\HA$.

To show the other inclusion,
let $\psi$ be a basic Nilsson solution of $\HrA$ in the direction
of $(0,w)$ corresponding to a standard pair that passes through zero,
with starting exponent $(\beta_0 - |v|, v)$.
We wish to prove that $\psi$ can be dehomogenized.
We can write
\[
\psi=x_0^{\beta_0-|v|}x^v 
\sum_{(-|u|,u) \in \ker_{\ZZ}(\rho(A))} x_0^{-|u|}x^u h_u(\log(x)),
\]
where $\beta_0-|v|$ is not an integer because $\beta_0$ is a
homogenizing value for $A$, $\beta$ and $w$.
Then we can perform
\[
\del_0^{|u|}\left(x_0^{\beta_0-|v|-|u|}x^{v+u}h_u(\log(x)) \right)= 
x_0^{\beta_0-|v|}x^{v+u}\widehat{p_u}(\log(x))
\]
and use this to define $\phi = \sum x^{v+u} p_u(\log(x))$ (with the same
relationship between $p$ and $\widehat{p}$ as in~\eqref{eq:qu}). We
claim that $\phi$ is
a basic Nilsson solution of $\HA$ and
$\psi = \rho(\phi)$. The proof of this claim is a reversal of the arguments in
Proposition~\ref{propo:Nilsson-to-Nilsson}. 
\end{proof}

\begin{definition}
If $\psi\in{\mathscr{N}}_{(0,w)}(\HrA)$ and $\phi\in\Fw(\HA)$ are
such that $\rho(\phi)=\psi$, we call 
$\phi$ a \textbf{dehomogenized Nilsson series}, or say that $\phi$ is the 
\textbf{dehomogenization} of $\psi$.  
\end{definition}

\section{Hyper\-geo\-metric Nilsson series for generic parameters}
\label{sec:generic-parameters}

When the parameter vector $\beta$ is sufficiently generic (see
Convention~\ref{remark:genericity of beta}),
the Nilsson solutions of $\HA$ are completely determined by the
combinatorics of the initial ideals of $I_A$. The goal of this section
is to study this case in detail.

In order to precisely describe the genericity condition used in
this section, we need to understand the initial ideal
$\inww(\HA)$ for generic parameter vectors.

\begin{lemma}
\label{l:genini}
For $\beta$ generic, 
\[
\inww(H_A(\beta)) = \inw(I_A) + \langle E - \beta \rangle.
\]
Therefore, all the fake exponents of $\HA$ with respect to $w$ are true exponents.
Moreover, under suitable genericity conditions for $\beta$, a better
description of $\inww(\HA)$ is available, namely
\[
\inww(\HA)=\tp(\inw(I_A))+\< E -\beta \>.
\]
\end{lemma}

\begin{proof}
This is a version of Theorems~3.1.3 and~3.2.11 in~\cite{SST}
for non homogeneous toric ideals.
The same proofs hold, since $I_A$ is always $A$-graded.
\end{proof}

\begin{convention}
\label{remark:genericity of beta}
In this section, we assume that $\beta$ is generic enough that
the second displayed formula in 
Lemma~\ref{l:genini} is satisfied, so that all exponents of $\HA$ with
respect to $w$
come from top-dimensional standard pairs of $\inw(I_A)$.

We also require that 
the only integer coordinates of these exponents are the ones imposed
by the corresponding standard pairs. In particular, the
exponents of $\HA$ with respect to $w$ have no negative integer coordinates. 

Finally, we ask that no two exponents differ by an integer vector.
Note that these integrality conditions force us to avoid
an infinite (but locally finite) collection of affine spaces.
\end{convention}

The following result holds without homogeneity
assumptions on the matrix $A$.

\begin{proposition}
\label{propo:phiv}
For any $v\in(\CC \minus \ZZ_{<0})^n$ such that $A\cdot v=\beta$, 
the formal series
\begin{equation}
\label{phisubv}
 \phi_v=\sum_{u\in \ker_{\ZZ}(A)}\frac{[v]_{u_-}}{[u+v]_{u_+}}x^{u+v}
\end{equation}
where
\begin{equation}
\label{eqn:phisubv-coeffs}
[v]_{u_-}=\prod_{u_i<0}\prod_{j=1}^{-u_i}(v_i-j+1); \quad
[u+v]_{u_+}=\prod_{u_i>0}\prod_{j=1}^{u_i}(v_i+j)
\end{equation}
is well defined and is annihilated by the hyper\-geo\-metric $D$-ideal
$\HA$. 

If $v$ is a fake exponent of $\HA$ corresponding to a standard
pair $(\del^a,\sigma)$ of $\inw(I_A)$, then $\phi_v$ is a basic Nilsson
solution of $\HA$.  
If moreover $\beta$ is a generic parameter vector, the
support of $\phi_v$ is the set 
\[
\supp(\phi_{v})
=\{u\in \ker_{\ZZ}(A) \mid u_i+v_i\geq 0 \quad \forall
i\notin\sigma\}.
\]
Thus,
\begin{equation}
\label{supp}
\phi_{v}
=\sum_{\{u\in \ker_{\ZZ}(A) \mid u_i+v_i\geq 0 \;\; \forall i\notin\sigma\}}
\frac{[v]_{u_-}}{[u+v]_{u_+}}x^{u+v}.
\end{equation}
\end{proposition}

\begin{proof}
This statement is essentially a combination of Proposition~3.4.1,
Theorem~3.4.2 and 
Lemma~3.4.6 in~\cite{SST}, which do not need homogeneity for $I_A$. It only
remains to be checked  
that $\phi_v$ is a basic Nilsson solution of $\HA$ in the direction of
$w$. To see this, note that 
$u\in\supp(\phi_v)$  implies $u\cdot w\geq 0$ as is shown in the proof 
of Theorem~3.4.2 in~\cite{SST}. 
Since $v$ is a fake exponent of $\HA$ with respect to any $w'$ in an open
neighborhood of $w$, 
we have in fact that $u\cdot w > 0$ for any nonzero $u$ in the support
of $\phi_v$. 
This argument shows that $\phi_v$ satisfies the first requirement from
Definition~\ref{defformal}. The remaining
conditions are readily verified.
\end{proof}

\subsection{A summary of known results for homogeneous $I_A$ and
  generic $\beta$} 
\label{subsec:homog}

In this brief subsection we recall combinatorial features and
convergence results for $A$-hyper\-geo\-metric functions under
the assumption that the toric ideal $I_A$ is homogeneous, which
implies that $\HA$ is regular holonomic. This material comes from~\cite{SST}.

For the purposes of this subsection only, we assume that 
the vector $(1, \dots, 1)$ is contained in the rational rowspan of $A$, 
or equivalently, that
the toric ideal $I_A$ is homogeneous with respect to the usual
$\ZZ$-grading in $\CC[\del]$. We keep our weight vector $w$, and we assume
the genericity on $\beta$ required in
Convention~\ref{remark:genericity of beta}. Later we apply what follows to
$\rho(A)$, $(\beta_0,\beta)$, and $(0,w)$.

If $v = \beta^{(\del^{\alpha},\sigma)}\in \CC^n$ is the exponent
of $\HA$ associated to a (top dimensional) standard pair $(\del^{\alpha},\sigma)$ of
$\inw(I_A)$, 
the series $\phi_v=\phi_{\beta^{(\del^{\alpha},\sigma)}}$ converges
in an open set $\mathscr{U}_{w, \be}$ of the
form~\eqref{eq:Uw} for some $\be \in \RR_{>0}^{n-d}$ 
by Theorem~2.5.16 in~\cite{SST}.

Let $\conv(A)$ be the convex hull of the columns $a_1,\dots,a_n$ of
$A$. 
Since the rational rowspan of $A$ contains $(1,\dots,1)$, there
is a hyperplane $H$ off the origin such that $\conv(A) \subseteq H$,
and as $A$ has full rank, $\conv(A)$ has dimension $d-1$.
The \emph{normalized volume} $\vol(A)$ is the Euclidean volume (in $H$) of
$\conv(A)$ normalized so that the unit simplex in the lattice $\ZZ A
\cap H$ has volume one. As we have assumed that $\ZZ A = \ZZ^d$, the
normalization is achieved by multiplying the Euclidean volume of
$\conv(A)$ by $(d-1)!$. 
We think of $A$, not just as a matrix, but as the point configuration 
$\{a_1,\dots,a_n\} \subset \ZZ^d$. A vector $w \in \RR_{>0}^{n}$
induces a subdivision $\Delta_w$ of the configuration $A$, by
projecting the lower hull of 
$\conv(\{(w_i,a_i) \mid i=1,\dots,n\})$ onto
$\conv(A)$ (see Chapter~8 of~\cite{Stu} for details). 
If $w$ is generic, $\Delta_w$ is a triangulation of $A$. Such
triangulations are usually called regular, but we use 
the alternative term \emph{coherent}. By Theorem~8.3 in~\cite{Stu},
the radical ideal $\sqrt{\inw(I_A)}$ is the Stanley--Reisner
ideal of the triangulation $\Delta_w$,
whose facet set is 
$
\{ \sigma \mid (\del^{\alpha},\sigma) \in \cT(\inw(I_A)) \text{ for
  some } \alpha \}.
$
We always write triangulations of $A$ as simplicial complexes
on $\{1,\dots,n\}$, but think of them geometrically: a simplex
$\sigma$ in such a triangulation corresponds to the geometric simplex
$\conv(\{a_i \mid i \in \sigma\})$.

The set
\begin{equation}
\label{eqn:basis}
\{ 
\phi_{\beta^{(\del^{\alpha},\sigma)}} \mid (\del^{\alpha},\sigma) \in
{\mathscr{T}}(\inw(I_A)) 
\}
\end{equation}
consists of 
$\rank(\inww(\HA))= \# \cT(\inw(I_A)) = \deg(I_A)=\vol(A)$
linearly independent series solutions of $\HA$ 
(linear independence follows from Lemma~\ref{thm:aux-ini-indep})
which have a common domain of convergence. In addition, the fact that
$I_A$ is homogeneous implies that $\HA$ is regular 
holonomic~\cite{H}, and therefore by Theorem 2.5.1 in~\cite{SST}, we have that
 $\rank(\inww(\HA)) = \rank(\HA)$. Thus
\eqref{eqn:basis} is, in fact, a basis for the (multivalued) holomorphic
solutions of $\HA$ in an open set of the form $\mathscr{U}_{w, \be}$ defined in~\eqref{eq:Uw} below.

\subsection{The general case}
\label{subsec:non homogeneous}

We now drop the homogeneity assumption on $I_A$, but keep the
genericity assumption for $\beta$. 
In this subsection, we describe the space $\Fw(\HA)$ using the
homogenization map $\rho$ defined in Section~\ref{sec:invert-der}.
An explicit basis for $\Fw(\HA)$ is constructed using 
the exponents of the ideal $\HA$. 
Our first step is to relate the exponents of $\HA$ with respect to $w$
to the exponents of $\HrA$ with respect to $(0,w)$.

The following is a well known result, whose proof we include for the
sake of completeness.

\begin{lemma}
\label{lema1}
Let $I \subset \CC[\del_1,\dots,\del_n]$ be an ideal and let
$\rho(I) \subset \CC[\del_0,\del_1,\dots,\del_n]$ be its
homogenization. 
Let $w \in \RR_{\geq 0}^n$ sufficiently generic so that $\inw(I)$
and $\inow(\rho(I))$ are both monomial ideals.
Suppose that 
\[
\inow(\rho(I))=\bigcap Q_i
\]
is a primary decomposition of the monomial ideal $\inow(\rho(I))$.
Then
\[
\inw(I)= \bigcap_{{\scriptstyle \del_0 \notin\sqrt{Q_i}}} 
\langle f(1,\del_1,\ldots,\del_n) \mid f\in Q_i\rangle
\]
is a primary decomposition of the monomial ideal $\inw(I)$.
\end{lemma}

\begin{proof}
Let $f \in I$ such that $\inw(f)$ is a monomial.
Then $\inow(\rho(f)) = \del_0^h\ \inw(f)$ for some
$h \in \NN$, with $h=0$ if $f$ is homogeneous.
Therefore $\inw(I)$ is obtained 
by setting $\del_0 \mapsto 1$ 
in the generators of $\inow(\rho(I))$. 
Now the result follows by observing that if $Q$ 
is primary monomial
ideal one of whose
generators is divisible by $\del_0$, then $Q$
must contain a power of $\del_0$ as a minimal generator.
\end{proof}

\begin{notation} \label{not:rhobeta0}
Given $\beta_0 \in \CC$ and $v \in \CC^n$, we set 
\[\label{eq:beta0}
\rho_{\beta_0}(v) = (\beta_0 -\sum_{i=1}^n v_i, v) = (\beta_0-|v|,v).\]
In particular,
$\rho_{0} = \rho$ maps $\ker_{\ZZ}(A)$ to $\ker_{\ZZ}(\rho(A))$. 
\end{notation}

\begin{lemma}
\label{rhobeta0}
Let $\beta_0, \beta$  generic and $w$ a weight vector for $\HA$. Then the map $v
\mapsto \rho_{\beta_0}(v)$ is a bijection between the set of exponents
of $\HA$  with 
respect to $w$ and the set of exponents of $H_{\rho(A)}(\beta_0, \beta)$ with respect
to $(0,w)$ associated to 
standard pairs that pass through zero.
\end{lemma}

\begin{proof}

By Lemma~\ref{lema1}, $v$ is the exponent of $\HA$
with respect to $w$ corresponding to a standard pair
$(\del^{\alpha},\sigma)$, if and only if $\rho_{\beta_0}(v)$ is the
exponent of $H_{\rho(A)}(\beta_0,\beta)$ with respect to
$(0,w)$
corresponding to the standard pair $(\del^{\alpha}, \{0 \} \cup \sigma)$.
\end{proof}

The following result is immediate.

\begin{lemma}
\label{imagenphiv}
Let $\beta\in \CC^d$ generic, $w$ a weight vector for $\HA$ and $\beta_0$ a homogenizing value for
$A$, $\beta$ and $w$. 
Let $v$ be an exponent of $\HA$ with respect to $w$ and consider the
map $\rho$ from
Proposition~\ref{propo:Nilsson-to-Nilsson}. Then
\[
  \rho(\phi_v)=\phi_{(\rho_{\beta_0} (v),v)},
\]
where $\phi_v, \phi_{(\rho_{\beta_0} (v),v)}$ are as in
Proposition~\ref{propo:phiv}. \qed
\end{lemma}

We now come to the main result in this section.

\begin{theorem}
\label{generadoresnilsson}
Let $\beta$ generic and $w$ a weight vector for $\HA$. 
Then 
\[
  \{\phi_v \mid v \ \textrm{is an exponent of} \ \HA  \
  \textrm{with respect to} \ w \}
\]
is a basis for $\Fw(\HA)$.
\end{theorem}

\begin{proof}
Fix a homogenizing value $\beta_0$ for $A$, $\beta$ and $w$, and 
let $\psi\in\Fw(\HA)$. Then, by Theorem~\ref{thm:image of rho},
$\rho(\psi) \in \cN_{(0,w)}(\HrA)$.
Since $I_{\rho(A)}$ is homogeneous, $\HrA$ is regular holonomic, and
we can use 
Lemma~\ref{imagenphiv} and the results from the previous subsection to
write   
\begin{equation}
\label{eqn:linear combination}
\rho(\psi)=\sum c_v\phi_{(\rho_{\beta_0} (v),v)}=\sum c_v\rho(\phi_v)
\end{equation}
where the sum is over the exponents of $\HrA$ with respect to $(0,w)$
corresponding to standard pairs that pass through zero, and the $c_v$ are
complex numbers.
By Lemma~\ref{rhobeta0}, the sum is over the exponents of
$\HA$ with respect to $w$. 
But $\rho$ is injective, so~\eqref{eqn:linear combination} implies
$\psi=\sum c_v\phi_v$.
Thus $\Fw(\HA)$ is contained in the $\CC$-span of the series $\phi_v$
associated to the exponents of $\HA$ with  
respect to $w$. Since the series $\phi_v$ are basic Nilsson solutions
of $\HA$ in the direction of $w$ (Proposition~\ref{propo:phiv}), the reverse
inclusion follows. 
Linear independence is proved using Lemma~\ref{thm:aux-ini-indep}.
\end{proof}

Recall that we have assumed that, if $w$ is a weight vector for
$\HA$, then $(0,w)$ is a weight vector for $\HrA$. In particular, this
implies that the subdivision of $\rho(A)$ induced by $(0,w)$ 
is in fact a triangulation (see Chapter~8 of~\cite{Stu}).

\begin{corollary}
\label{coro:dimnilsson}
Assume $\beta \in\CC^d$ is generic and let $w$ be a weight vector for $\HA$.
The dimension of the space of Nilsson solutions of $\HA$ in the direction of
$w$ is 
\[
\dim_{\CC}(\Fw(\HA)) = 
\sum_{\tiny{\begin{array}{c}
   \sigma  \textrm{ facet of } \Delta_{(0,w)} \\
  \textrm{ such that } 0 \in \sigma \end{array}}} 
\vol(\sigma) .
\]
\end{corollary}

\begin{proof}
By Theorem~\ref{generadoresnilsson}, the number $\dim_{\CC}(\Fw(\HA))$ is the
number of exponents of $\HA$ with respect to $w$, which is the number
of top-dimensional standard pairs of $\inw(I_A)$, because $\beta$ is
generic. Using the bijection 
from Lemma~\ref{rhobeta0}, we conclude that $\dim_{\CC}(\Fw(\HA))$ is
the number of top-dimensional standard pairs of $\inow(I_{\rho(A)})$
that pass through zero. Given $\sigma \subset \{0,1,\dots,n\}$ of
cardinality $d+1$ such that $0\in \sigma$, the number of
top-dimensional standard pairs of 
$\inow(I_{\rho(A)})$ of the form $(\del^{\alpha}, \sigma)$
is the multiplicity of $\<\del_i \mid i \notin \sigma \>$ as an
associated prime of $\inow(I_{\rho(A)})$ by Lemma~3.3 in~\cite{STV}. 
This number equals the normalized volume of the simplex
$\{0\} \cup \sigma$ by Theorem~8.8 in~\cite{Stu}, and the result follows.
\end{proof}

\begin{proposition}
\label{l:distr}
Suppose that $\beta$ is generic, and let $w$ be a weight vector for $\HA$. Then
\[
\rank(\inww(\HA) = \deg(\inw(I_A)).
\]
\end{proposition}

\begin{proof}
By Lemma~\ref{l:genini},  
$\inww(H_A(\beta)) = \tp(\inw(I_A)) + \langle E - \beta
\rangle$. 
Since $\tp(\inw(I_A))$ is a monomial ideal, the $D$-ideals $\inww(\HA)$ and 
$\< x^u \del^u \mid \del^u \in \tp(\inw(I_A)) \> + \< E-\beta \>$ have
the same holomorphic solutions.

We denote $x_i \del_i = \theta_i$, and observe that
$\CC[\theta_1,\dots,\theta_n]$ is a commutative polynomial subring of
$D$. Also recall that the Euler operators $E_i-\beta_i$ belong to
$\CC[\theta]$. Since $x^u\del^u =
\prod_{i=1}^n\prod_{j=0}^{u_i-1}(\theta_i-j)$, Proposition~2.3.6 in~\cite{SST} 
can be applied to conclude that the holonomic rank $\rank(\inww(\HA))$ equals
\begin{equation}
\label{eqn:dimension of frobenius}
 \dim_{\CC} \bigg(\frac{\CC[\theta]}{ \<
\prod_{i=1}^n\prod_{j=0}^{u_i-1}(\theta_i-j)\mid \del^u
\in\tp(\inw(I_A)) \> + \<E-\beta\>} 
\bigg).
\end{equation}
Considered as a system of polynomial equations in $n$ variables $(\theta_1, \dots, \theta_n)$, 
the zero set of the ideal
$\< \prod_{i=1}^n\prod_{j=0}^{u_i-1}(\theta_i-j)\mid \del^u
\in\tp(\inw(I_A)) \>$ is a  subvariety of $\CC^n$ consisting of
$\deg(\inw(I_A))$ irreducible 
components, each of which is a translate of a $d$-dimensional
coordinate subspace of 
$\CC^n$. By Corollary~3.2.9 in~\cite{SST}, each of these
components meets the zero set of $\<E-\beta\>$ in exactly one point.
Therefore,~\eqref{eqn:dimension of frobenius} equals
$\deg(\inw(I_A))$, and the proof is complete. 
\end{proof}

\begin{corollary}
\label{propo:degree computation}
Let $\beta$ generic and $w$ a weight vector for $\HA$.
Then
\begin{equation}\label{volumen}
\rank(\inww(\HA)) =
\sum_{\tiny{\begin{array}{c} \sigma \textrm{ facet of } \Delta_{(0,w)}\\ 
\textrm{such that } 0 \in \sigma\end{array}}} \vol(\sigma),
\end{equation}
\end{corollary}

\begin{proof}

We need to show that the sum on the right hand side of~\eqref{volumen}
equals $\deg(\inw(I_A))$. By Lemma~3.3 in~\cite{STV}, this degree
equals the number of 
top-dimensional standard pairs of $\inw(I_A)$, 
which equals the number of top-di\-men\-sion\-al standard pairs of
$\inow(I_{\rho(A)})$ passing through zero by Lemma~\ref{rhobeta0}. As
in the proof of Corollary~\ref{coro:dimnilsson}, the number of such
standard pairs is the desired sum.
\end{proof}

\begin{corollary}
\label{coro:dimequalrankini}
Suppose that $\beta$ is generic and $w$ is a weight vector for $\HA$. Then
\[
 \dim_{\CC}(\Fw(\HA))= \rank(\inww(\HA)).
\]
\end{corollary}

\begin{proof}
Immediate from Corollary~\ref{coro:dimnilsson} and 
Corollary~\ref{propo:degree computation}. 
\end{proof}

The following corollary states that, for certain weight vectors, the
dimension of the space $\Fw(\HA)$ equals 
$\rank(\HA)$. However, this fails in general, as
Example~\ref{ejem:dimnilmenorrank} shows. 
This means that, as expected, formal Nilsson series are not enough to 
understand the solutions of irregular hyper\-geo\-metric systems.

When working with matrices $A$ whose
columns are not assumed to lie in a hyperplane off the origin,
$\vol(A)$ denotes the normalized volume of the convex hull of
$\{0,a_1,\dots,a_n\}$ with respect to the lattice $\ZZ A = \ZZ^d$.

\begin{corollary}
\label{coro:notalways}
Suppose that $\beta$ is generic, and $w$ is a weight vector for $\HA$.
The equality
\[
\dim_{\CC}(\Fw(\HA))=\vol(A) =\rank(\HA)
\]
holds if and only if \ $0$ belongs to every maximal simplex in the
triangulation $\Delta_{(0,w)}$ of $\rho(A)$. 
\end{corollary}

\begin{proof}
Note that $\vol(A)=\vol(\rho(A))$, which is the sum of
the volumes of all the maximal simplices in $\Delta_{(0,w)}$.
Therefore $\dim_{\CC}(\Fw(\HA))=\vol(A)$ if and only if all maximal
simplices in $\Delta_{(0,w)}$ pass through zero. Now use a result of
Adolphson~\cite{Ado} that, for generic $\beta$, $\rank(\HA)=\vol(A)$.
\end{proof}

\begin{example}
\label{ejem:dimnilmenorrank}
Let
\[
A=\left[\begin{array}{ccc}
       1 & 0 & 1\\
       0 & 1 & 1  
      \end{array}\right] \ ,\;  \textrm{ so that } \;
    I_A=\langle\partial_3-\partial_1\partial_2\rangle \; .
\]
Then $\vol(\rho(A)) = \vol(A) = 2$ is the generic rank of 
$\HA$ and of $H_{\rho(A)}(\beta_0,\beta)$.

If $w$ is a perturbation of $(1,1,1)$, we have
$\inw(I_A)=\< \partial_1\partial_2 \>$ and the corresponding
triangulation is
$\Delta_{(0,w)}=\{\{0,1,3\},\{0,2,3\}\}$. In this case 
\[
\dim_{\CC}(\Fw(\HA))=\vol(\{0,1,3\})+\vol(\{0,2,3\}) = 2 =\rank(\HA).
\]
On the other hand, if $w$ is a perturbation of $(1,1,3)$, then we have 
$\inw(I_A)=\< \partial_3 \>$ and the corresponding triangulation is
$\Delta_{(0,w)}=\{\{0,1,2\},\{1,2,3\}\}$. In this case 
\begin{equation}
\dim_{\CC}(\Fw(\HA))=\vol(\{0,1,2\})= 1 <\rank(\HA).
\end{equation}
\end{example}

\section{Logarithm-free Nilsson series}
\label{sec:logarithm free}

If we assume that $\beta$ is generic, then all of the Nilsson
solutions of $\HA$ are automatically logarithm-free. We now turn our
attention to the logarithm-free Nilsson solutions of $\HA$ without any
assumptions on the parameter $\beta$.

\begin{definition}
For a vector $v\in\CC^n$, its \textbf{negative support} is
the set of indices  
\[
\nsup(v) = \big\{ i \in \{1,\dots, n\} \mid v_i\in\ZZ_{<0} \big\}.
\]
A vector $v\in \CC^n$ has \textbf{minimal negative support}
if $\nsup(v)$ does not properly contain $\nsup(v+u)$ for any nonzero $u \in
\ker_{\ZZ}(A)$. 
We denote 
\begin{equation}
\label{eqn:Nv}
N_v = \{ u \in \ker_{\ZZ}(A) \mid \nsup(u+v) = \nsup(v) \}.
\end{equation}
\end{definition}

When $\beta$ is arbitrary, the fake exponents of
$\HA$ with respect to a weight vector $w$ can have negative integer 
coordinates. For such a $v$, we wish to construct an associated basic
Nilsson solution of $\HA$, in the same way as we did in
Proposition~\ref{propo:phiv}.

\begin{proposition}
\label{propo:phisubvlibrelog}
Let $w$ be a weight vector for $\HA$ and let $v\in\CC^n$ be a fake
exponent of $\HA$ with respect to $w$.
The series
\begin{equation}
\label{phisubvlibrelog}
\phi_v=\sum_{u\in N_v}\frac{[v]_{u_-}}{[u+v]_{u_+}}x^{u+v} \ ,
\end{equation}
where $[v]_{u_-}$ and $[u+v]_{u_+}$ are as in (\ref{phisubv}), is well
defined. This series
is a formal solution of $\HA$ if and only if $v$ has minimal negative support,
and in that case, $\phi_v$ is a basic Nilsson solution of $\HA$ in the
direction of $w$. Consequently, $v$ is an exponent of $\HA$ with respect
to $w$.
\end{proposition}

\begin{proof}
The series is well defined because, as the summation is over $N_v$, there
cannot be any zeros in the denominators of the summands. The second assertion
holds with the same proof as Proposition 3.4.13 of~\cite{SST}.
To see that $\phi_v$ is a basic Nilsson solution of $\HA$, we can argue in the
same way as in the proof of
Proposition~\ref{propo:phiv}.
\end{proof}

Lemma~3.4.12 of~\cite{SST} shows that if the negative support of $v$ is empty, 
then the series \eqref{phisubvlibrelog} and \eqref{supp} coincide.

We now consider Nilsson series in the direction of a weight vector.

\begin{definition}
\label{def:directionw}
Formal solutions of $\HA$ of the form~\eqref{eqn:Nilsson series} that satisfy
the first two conditions in Definition~\ref{defformal} are called \textbf{(formal)
series solutions of $\HA$ in the direction of $w$}. 
\end{definition}

The $\CC$-vector space of
\emph{logarithm-free} formal $A$-hyper\-geo\-metric series with
parameter $\beta$ in the direction
of $w$ is denoted by $\cS_w(\HA)$.  

\begin{theorem}
\label{prop:bSaito}
Let $w$ be a weight vector for $\HA$. The set
\begin{equation}
\label{eqn:basis log free}
\{ \phi_v \mid v \textrm{ is an exponent of } \HA \textrm{ with
  minimal negative support}\}
\end{equation}
is a basis for $\cS_w(\HA)$.
\end{theorem}

The previous result was stated in Display~(7) of~\cite{S}, 
in the special case when $I_A$ is
homogeneous. Its proof for the case when $\beta \in \ZZ^d$ appeared 
in Proposition~4.2 of~\cite{CDRV}; we generalize that argument here.

\begin{proof}
Linear independence of the proposed basis elements follows from
Lemma \ref{thm:aux-ini-indep}, 
so we need only show that these series span $\cS_w(\HA)$.

Let $G(x) \in \cS_w(\HA)$, and suppose that $x^{\nu}$ appears in $G$ with
nonzero coefficient $\lambda_{\nu} \in \CC$. 
We claim that $\nu$ has minimal negative support.
By contradiction, let $u \in \ker_{\ZZ}(A)$ such that $\nsupp(\nu+u)$ is
strictly contained in $\nsupp(\nu)$. This means that there is $1 \leq i
\leq n$ such that $\nu_i \in \ZZ_{<0}$ and $\nu_i+u_i \in \NN$. In
particular, $u_i>0$.

Since $G$ is a solution of $\HA$,
the operator $\del^{u_+}-\del^{u_-} \in I_A$ annihilates $G$. Note that
$\nsupp(\nu+u) \subset \nsupp(\nu)$ implies that 
$\del^{u_-} x^{\nu} \neq 0$. 
Then some term from $\del^{u_+} G$ needs to equal 
$\lambda_v \del^{u_-} x^{\nu}$, 
which is a nonzero multiple of $x^{\nu-u_-}$. But any function $f$
such that $\del^{u_+} f = x^{\nu-u_-}$ must involve $\log(x_i)$. This
produces the desired contradiction.

Fix $\nu$ such that $x^{\nu}$ appears with nonzero coefficient $\lambda_{\nu}$
in $G$, and let $\psi$ be the subseries of $G$ consisting of terms of
the form $\lambda_{\nu+u} x^{\nu+u}$ with $u \in \ker_{\ZZ}(A)$ and
$\lambda_{\nu+u}\in \CC$, such that $\nsupp(\nu+u)=\nsupp(\nu)$. 
Our goal is to show that $\psi$ is a constant multiple of one of the series
from~\eqref{eqn:basis log free}. This will conclude the proof.

We claim that $\psi$ is a solution of $\HA$. That the Euler
operators $\<E-\beta\>$ annihilate $\psi$ follows since they
annihilate every term of $G$. To deal with the toric operators,
recall that
$\del^{u_+} G = \del^{u_-} G$ for all $u \in \ker_{\ZZ}(A)$. But terms
in $\del^{u_+} G$ that come from $\psi$ can only be matched by terms
in $\del^{u_-} G$ that also come from $\psi$, so
$\del^{u_+}-\del^{u_-}$ must annihilate $\psi$, for all $u \in \ker_{\ZZ}(A)$.

Since $G$ is a solution of $\HA$ in the direction of $w$, so is $\psi$.
Therefore $\inw(\psi)$ is a logarithm-free solution of
$\inww(\HA)$. 
Theorems~2.3.9 and~2.3.11 in~\cite{SST} imply that $\inw(\psi)$
is a linear combination of (finitely many) monomial 
functions arising from exponents of $\HA$ with
respect to $w$. By construction of $\psi$, these exponents differ by
elements of $\ker_{\ZZ}(A)$.
Arguing as in the proof of Theorem~3.4.14 in~\cite{SST}, we see that
$\inw(\psi)$ can only have 
one term, that is,
$\inw(\psi) = \lambda_{v} x^{v}$, where
$v$ is an exponent of $\HA$ with respect to $w$ and 
$\lambda_{v} \neq 0$. 
Since $\lambda_{v} x^{v}$ is a term in $G$, $v$ has
minimal negative support. Thus, $v$ is an exponent of $\HA$ with
respect to $w$ that has minimal negative support.

To finish the proof, we show that $\psi = \lambda_{v} \phi_{v}$.
Suppose that $u \in N_{v}$, which means that $u \in \ker_{\ZZ}(A)$
and $\nsupp(v+u)=\nsupp(v)$. The equality of negative supports implies that 
$\del^{u_-} x^{v} = [v]_{u_-}x^{v-u_-}$ is nonzero. Since
$\del^{u_-}\psi=\del^{u_+}\psi$, $\del^{u_+}\psi$ must contain the
term $\lambda_{v} [v]_{u_-}x^{v-u_-}$, which can only come from
$\del^{u_+}\lambda_{v+u}x^{v+u}$. Thus
\[
\lambda_{v} [v]_{u_-} x^{v-u_-} = 
\del^{u_+} \lambda_{v+u}x^{v+u} = 
\lambda_{v+u} [v+u]_{u_+} x^{v+u-u_+}.
\]
Consequently
$\lambda_{v+u} = \lambda_{v} \frac{[v]_{u_-}}{[v+u]_{u_+}}$, that is, the coefficient of
$x^{v+u}$ in $\psi$ equals $\lambda_{v}$ times the coefficient of
$x^{v+u}$ in $\phi_{v}$. Therefore 
$\psi = \lambda_{v}\phi_{v}$, as we wanted.
\end{proof}

The next theorem gives a bijective map between the space of logarithm-free series
solutions of $\HA$ in the direction of $w$ and a subspace of the
logarithm-free solutions of $\HrA$ in the direction of $(0,w)$. Note
that $\cS_w(\HA) \subseteq \Fw(\HA)$ follows immediately from
the previous result, as the series $\phi_v$ are basic Nilsson
solutions of $\HA$ in the direction of $w$.

\begin{theorem}
\label{teo:map-of-fake-exps}
Let $w$ be a weight vector for $\HA$ and let $\beta_0$ be a
homogenizing value for  $A$, $\beta$, and $w$.
Then $\rho(\cS_w(\HA))$ equals the
$\CC$-linear span of the logarithm-free basic Nilsson solutions of
$\HrA$ in the direction of $(0,w)$
which are associated to standard pairs of $\inow(I_{\rho(A)})$ that
pass through zero. 
\end{theorem}

\begin{proof}
Since $\rho$ is linear,
we only need to consider the image of the elements of a basis for the space
$\cS_w(\HA)$, such as the one given in Theorem~\ref{prop:bSaito}.

We claim that if $v$ is a fake exponent of $\HA$ with respect to $w$
that has minimal negative support,
then $\rho_{\beta_0}(v) = (\beta_0 - |v|, v)$ is a fake exponent of
$\HrA$ with respect to $(0,w)$ corresponding to a standard pair that
passes though zero, and moreover, $\rho_{\beta_0}(v)$ has minimal
negative support. The first part is proved using Lemma~\ref{lema1}. To
see that $\rho_{\beta_0}(v)$ has minimal negative support, first
recall that $\beta_0-|v| \notin \ZZ$ because $\beta_0$ is a
homogenizing value for $A$, $\beta$ and $w$. This implies that $0$ is
not in the negative support of $\rho_{\beta_0}(v) + \mu$, for any
$\mu \in \ker_{\ZZ}(\rho(A))$. Now use the bijection $u \mapsto (-|u|,u)$
between $\ker_{\ZZ}(A)$ and $\ker_{\ZZ}(\rho(A))$ and the fact that
$v$ has minimal negative support, to conclude that $\rho_{\beta_0}(v)$
has minimal negative support.

To complete the proof, we show that 
$\rho(\phi_v) = \phi_{\rho_{\beta_0}(v)}$.
Lemma~\ref{imagenphiv} is this statement in the case when 
$\nsup(v) = \emptyset$, but now we have to pay attention to the
supports of these series. 

The same argument we used to check that $\rho_{\beta_0}(v)$ has
minimal negative support yields $N_v=\pi(N_{\rho_{\beta_0}(v)})$,
where $\pi$ is the projection onto the last $n$ coordinates, and therefore
$\rho(\phi_v)$ and $\phi_{\rho_{\beta_0}(v)}$ have the same support.
The verification that the corresponding coefficients are the same is
straightforward. 
\end{proof}

\section{Convergence of hyper\-geo\-metric Nilsson series}
\label{sec:convergence}

Until now, we have made no convergence considerations in our study of 
Nilsson solutions of $A$-hyper\-geo\-metric systems. 
The purpose of this section is to investigate
convergence issues in detail. In particular, 
Theorem~\ref{ultimoteorema} states 
that, if $w$ is a perturbation of $(1,\dots,1)$, the elements of $\Fw(\HA)$ have a
common domain of convergence. Moreover, assuming that the cone spanned by the
columns of $A$ is strongly convex, results from 
Section~\ref{sec:Hotta} imply that
$\dim_{\CC}(\Fw(\HA)) = \rank(\HA)$. This provides an explicit
construction for the space of (multivalued) holomorphic solutions of $\HA$ 
in a particular open subset of $\CC^n$.

When the parameter $\beta$ is generic and $w$ is a
perturbation of $(1,\dots,1)$, the convergence of the elements of
$\Fw(\HA)$ was shown in~\cite{OT}. In Subsection~\ref{subsec:generic
  convergence} we complete this study by considering other weight
vectors.

\begin{notation}\label{not:basis of kernel}
We have already used the notation $| \cdot |$ to mean the coordinate
sum of a vector. When applied to a monomial, such as $x^u$, $| \cdot|$ 
means complex absolute value. 
Let $w$ be a weight vector for $\HA$ and let
$\left\{ \gamma_1,\ldots, \gamma_{n-d} \right\}\subset\ZZ^n$ 
be a $\ZZ$-basis for $\ker_{\ZZ}(A)$ such that $\gamma_i \cdot w > 0$ for
$i=1,\ldots,n-d$. For any $\be = (\varepsilon_1, \dots,
\varepsilon_{n-d}) \in \RR_{>0}^{n-d}$, we define the (non empty) open set
\begin{equation}\label{eq:Uw}
\mathscr{U}_{w, \be} =
\left\{ x \in \CC^n \mid 
|x^{\gamma_{i}}| < \varepsilon_i  \ \text{ for } \ i=1,\ldots,n-d \right\}.
\end{equation}
\end{notation}

\subsection{General parameters} 
\label{subsec:general convergence}

The following result is the main technical tool in this section.

\begin{theorem}
\label{ultimolema}
Let $w$ be a weight vector for $\HA$ and let
$\{ \gamma_1,\dots, \gamma_{n-d} \} $ be a basis for
$\ker_{\ZZ}(A)$ such that $w \cdot \gamma_i > 0$ for $i=1,\dots,n-d$.
Let $\phi= \sum x^{v+u}p_u(\log(x))$ be a basic Nilsson solution of
$\HA$ in the direction of $w$ as
in~\eqref{eqn:Nilsson series}, such that $|u| \geq 0$ for almost all
$u \in \supp(\phi)$, meaning that the set 
$\{ u \in \supp(\phi) \mid |u| < 0 \}$ 
is finite. Then there exists $\be \in \RR_{>0}^{n-d}$ 
such that $\phi$ converges in the open set
$\mathscr{U}_{w, \be}$.
\end{theorem}

\begin{proof}
We may assume without loss of generality that $|u|\geq0$ for all
$u\in\supp(\phi)$. Choose $\beta_0$ a homogenizing value for $A$,
$\beta$ and $w$, and recall
from Definition~\ref{def:homog-sol} 
that the homogenization of $\phi$ is
\[ 
\rho(\phi) = 
\sum_{u\in\supp(\phi)} \del_0^{|u|} \left[ x_0^{\beta_0 -|v|} x^{v+u} 
         \widehat{p_u}(\log(x_0),\log(x_1),\dots,\log(x_n)
  )\right],
\]
where $p_u$ and $\widehat{p_u}$ are related by \eqref{eq:qu}.
By Theorem~\ref{propo:Nilsson-to-Nilsson}, 
$\rho(\phi)$ is a basic Nilsson solution of
$\HrA$ in the direction of $(0,w)$. Since $\HrA$ has regular
singularities, Theorem~2.5.16 in~\cite{SST} implies that there exists
$\be \in \RR_{>0}^{n-d}$ such that 
$\rho(\phi)$ converges (absolutely) in the open set
\[
\mathscr{U}_{(0,w), \be} = \{ (x_0,x) \in \CC^{n+1} \mid 
|x_0^{-|\gamma_i|} x^{\gamma_i}| < \varepsilon_i , i = 1, \dots, n-d\} .
\]
We make use of the convergence of
$\rho(\phi)$ to prove convergence for $\phi$.

As $\rho(\phi)$ converges absolutely in $\mathscr{U}_{(0,w), \be}$,
convergence is preserved when we reorder terms. Use the fact that $|u|
\geq 0$ for all $u \in \supp(\phi)$ to rewrite
 \[
\rho(\phi) 
        = \sum_{m=0}^{\infty} \del_0^m  \left[ x_0^{\beta_0 - |v|} f_m\right]
\]
where 
\[
f_m(x_0,\dots,x_n)= \sum_{|u|=m}x^{v+u}
\widehat{p_u}(\log(x_0),\log(x_1),\dots,\log(x_n) )
\]
is a polynomial in $\log(x_0)$ whose coefficients are (multivalued) holomorphic
functions of the $n$ variables $x_1,\dots,x_n$. 
Recall that, by Definition~\ref{defformal}, there exists a
positive integer $K$ such that 
the degree of $f_m$ in $\log(x_0)$ is less than or equal to $K$ for
all $m \in \NN$.
A key observation is that
\begin{equation}\label{eqn:specialize to 1}
\sum_{m=0}^{\infty} \left. \left(x_0^{\beta_0-|v|-m}
f_m(x_0,x_1,\dots,x_n)\right) \right|_{x_0=1}=\phi(x_1,\ldots,x_n).
\end{equation}
Since 
$\{1\} \times \mathscr{U}_{w,\be} \subset \mathscr{U}_{(0,w),\be}$,
if we show that 
$\sum_{m=0}^{\infty}x_0^{\beta_0-|v|-m} f_m$ converges absolutely on 
$\mathscr{U}_{(0,w),\be}$,
the convergence of $\phi$ on $\mathscr{U}_{w,\be}$ will follow.

For $\lambda \in \CC$ and $m \in \NN$, we denote the $m$-th descending
factorial by 
\[
[\lambda]_m = \lambda(\lambda-1)\ldots(\lambda-m+1).
\]
Set $\lambda = \beta_0-|v|$. Since $\beta_0$ is a homogenizing
value for $A$, $\beta$ and $w$, we have $\lambda \notin\ZZ$.
We claim that the domain of convergence of
$\sum_{m=0}^{\infty}[\lambda]_m x_0^{\lambda-m}f_m$ contains
$\mathscr{U}_{(0,w), \be}$.
But if this is true, the convergence of
$\sum_{m=0}^{\infty}x_0^{\lambda-m} f_m$ on $\mathscr{U}_{(0,w), \be}$
follows by comparison, since the absolute
value of $[\lambda]_m$ grows like $(m-1)!$ as $m$ goes to $\infty$. 
Thus, all we need to show in order to finish our proof is that 
$\sum_{m=0}^{\infty}[\lambda]_m x_0^{\lambda-m}f_m$ converges 
absolutely on $\mathscr{U}_{(0,w), \be}$. 

Consider $f_m$ as a polynomial in $\log(x_0)$. 
By construction, the coefficients of $f_m$ are constant with
respect to $x_0$. Denote by $f_m^{(r)}$ the $r$-th derivative of $f_m$ with
respect to $\log(x_0)$. Then $f_m^{(K+1)}=0$ since the degree of $f_m$
in $\log(x_0)$ is at most $K$.
We compute $\del_0^{m}(x_0^{\lambda}f_m)$ for $m \geq K$, using the fact that
$f_m^{(r)}= 0$ if $r>K$.
\begin{align*}
\del_0^{m}(x_0^{\lambda}f_m) 
 & =\del_0^{m-1}\del_0(x_0^{\lambda}f_m) 
   =\del_0^{m-1}(x_0^{\lambda-1}(\alpha f_m+f_m')) \\
 &
 =\del_0^{m-2}(x_0^{\lambda-2}
(\lambda(\lambda-1)f_m+(\lambda+(\lambda-1))f_m'+f_m'')) \\
 & =\cdots \\
 &
=x_0^{\lambda-m}(c_0(\lambda,m)f_m+c_1(\lambda,m)f_m'+\ldots+c_K(\lambda,m)f_m^{
(K)}),
\end{align*}
where 
{\small\[
c_j(\lambda,m) = \sum_{1 \leq i_1 < \dots < i_{m-j} \leq m} \prod_{k=1}^j
(\lambda-i_k+1).
\]}
Note that, as $m$ goes to $\infty$, the dominant term in absolute
value in $c_j(\lambda,m)$ is $\prod_{r=j}^{m-1}(\lambda-r+1)$,
which grows like $\prod_{r=j}^{m-1}r=\frac{(m-1)!}{(j-1)!}$. 
But then, if $j>0$, $[\lambda-j]_m$ grows faster than $c_j(\alpha,m)$ as $m$
goes
to $\infty$, because $[\lambda-j]_m$ grows like $\frac{(m+j-1)!}{(j-1)!}$. In
other words,
\begin{equation}\label{dos}
\lim_{m\rightarrow\infty}\frac{c_j(\lambda,m)}{[\lambda-j]_m}=0 \quad
\text{for } j\geq1 .
\end{equation}
Since $\rho(\phi)$ converges absolutely on the open set 
$\mathscr{U}_{(0,w),\be}$,
$\del_0\rho(\phi)$ is also absolutely convergent on 
$\mathscr{U}_{(0,w),\be}$, and
\begin{align*}
\del_0\rho(\phi) 
   &=\del_0\sum_{m=0}^{\infty} \del_0^m \left[ x_0^{\lambda} f_m\right] 
    =\sum_{m=0}^{\infty} \del_0^m \left[\del_0 x_0^{\lambda} f_m\right] \\
   &=\sum_{m=0}^{\infty} \del_0^m \left[\lambda x_0^{\lambda-1}
     f_m+x_0^{\lambda}f_m'\right]x_0^{-1} \\
   &=\lambda \sum_{m=0}^{\infty} \del_0^m \left[x_0^{\lambda-1} f_m\right]
+\sum_{m=0}^{\infty} \del_0^m\left[x_0^{\lambda-1}f_m'\right].
\end{align*}
The series
$\sum_{m=0}^{\infty} \del_0^m \left[x_0^{\lambda-1}  f_m\right]$ 
converges in $\mathscr{U}_{(0,w), \be}$ because it is a
basic Nilsson solution of the regular hyper\-geo\-metric system
$H_{\rho(A)}(\beta_0-1,\beta)$ in the direction of $(0,w)$. (We may need
to decrease $\be$ coordinatewise for the previous assertion to hold.)
This, and the convergence of $\del_0(\rho(\phi))$, imply that 
$\sum_{m=0}^{\infty} \del_0^m\left[x_0^{\lambda-1}f_m'\right]$
converges absolutely on $\mathscr{U}_{(0,w), \be}$. Proceeding by
induction, we conclude that
\begin{equation}\label{cuatro}
\sum_{m=0}^{\infty} \del_0^m x_0^{\lambda-j} f_m^{(j)} \quad
\text{converges absolutely on } \mathscr{U}_{(0,w), \be} \text{ for }
j=1,\dots,K.
\end{equation}
Now we induct on $K-\ell$ to show that 
\begin{equation}\label{doce}
\sum_{m=0}^{\infty}x_0^{\lambda-m-\ell}[\lambda-\ell]_mf_m^{(\ell)} \quad
\text{converges absolutely on } \mathscr{U}_{(0,w), \be} \text{ for }
\ell=0,1,\ldots,K.
\end{equation}
The $\ell = 0$ case of this assertion is exactly what we needed to
verify in order to finish the proof.

If $K - \ell = 0$, then $f_m^{(\ell)} = f_m^{(K)}$ is a (maybe zero)
constant with respect to $x_0$, and therefore~\eqref{doce} is the
$j=K$ case of~\eqref{cuatro}.
For the inductive step, we compute the $m$-th derivative inside the series.
The sum $\sum_{m=0}^{\infty} \del_0^m x_0^{\lambda-\ell} f_m^{(\ell)}$ equals
{\small
\begin{align*}
  {} &      \sum_{m=0}^{\infty}x_0^{\lambda-\ell-m}([\lambda-\ell]_m
f_m^{(\ell)}+c_1(\lambda-\ell,m) f_m^{(\ell+1)} +
     \cdots+c_{K-\ell}(\lambda-\ell,m)f_{m}^{(K)})  \\
&  = 
    \sum_{m=0}^{\infty}x_0^{\lambda-\ell-m}[\lambda-\ell]_m
    f_m^{(\ell)}  \\
& \qquad +
    \sum_{m=0}^{\infty}x_0^{\lambda-\ell-m}c_1(\lambda-\ell,m)
    f_m^{(\ell+1)} + \cdots + \sum_{m=0}^{\infty}x_0^{\lambda-\ell-m}
    c_{K-\ell}(\lambda-\ell,m)f_{m}^{(K)}. 
\end{align*}}
We want to show that 
$\sum_{m=0}^{\infty}x_0^{\lambda-\ell-m}[\lambda-\ell]_mf_m^{(\ell)}$
converges absolutely on $\mathscr{U}_{(0,w),\be}$.
We know that $\sum_{m=0}^{\infty} \del_0^m x_0^{\lambda-\ell}
f_m^{(\ell)}$ converges absolutely on $\mathscr{U}_{(0,w),\be}$
by~\eqref{cuatro}, so we need to control the other summands.
But the inductive hypothesis tells us that~\eqref{doce}
is true for $\ell+1,\dots,K$. By comparison using~\eqref{dos}, and
harmlessly multiplying by $x_0^{j}$, we conclude that the series
$\sum_{m=0}^{\infty}x_0^{\lambda-\ell-m}c_j(\lambda-\ell,m)
f_m^{(\ell+j)}$ converges absolutely on $\mathscr{U}_{(0,w),\be}$ 
for $1\leq j<K-\ell$.
\end{proof}

\begin{corollary}\label{coro:convergence-of-solutions-from-111}
If $w$ is a weight vector for $H_A(\beta)$ which is a perturbation of
$(1,\dots,1)$, then all the basic Nilsson solutions of $H_A(\beta)$ in
the direction of $w$ have a common (open) domain of convergence. 
\end{corollary} 

\begin{proof}
Let $w$ be a weight vector for $\HA$ which is a perturbation of
$(1,\dots,1)$. Recall that, if 
$\phi=x^v \sum_{u\in C} x^u p_u(\log(x))$ 
is a basic Nilsson solution of
$H_A(\beta)$ in the direction of $w$, then $|u|\geq 0$ for $u \in
C$. (See the paragraph after Definition~\ref{def:perturbation}.) 
Now apply Theorem~\ref{ultimolema}.
\end{proof}

One of the main objectives of this article was to construct a basis of
series solutions of $\HA$ that have a common domain of
convergence. While such constructions are well known in the regular
case, when $I_A$ is inhomogeneous, important theoretical tools
become unavailable. A way of bypassing this difficulty is to assume that
the parameters are generic and $w$ is a perturbation of $(1,\dots,1)$
as in~\cite{OT}. The following result gives 
the desired construction, without any assumptions on $\beta$ (but with
the same assumption on $w$).
Its proof can be found in
Section~\ref{sec:Hotta}, after Theorem~\ref{thm:rank-lifting}.

\begin{theorem}
\label{ultimoteorema}
Assume that the cone over the columns of $A$ 
is strongly convex, and let  
$w$ be a weight vector for $\HA$ that is a perturbation of
$(1,\ldots,1)$. Then 
\[
\dim_{\CC}(\Fw(\HA)) = \rank(\HA)
\] 
and
there exists $\be \in \RR_{>0}^{n-d}$ such that
every element of $\Fw(\HA)$ converges in the open set $\cU_{w,\be}$.
\end{theorem}

\subsection{Generic parameters} 
\label{subsec:generic convergence}

In this subsection, we assume that $\beta$ is generic as in
Convention~\ref{remark:genericity of beta}.
In this case, by Theorem~\ref{generadoresnilsson}, the set
\begin{equation}
\label{eqn:Bw}
\cB_w=
\{
\phi_v \mid v  \textrm{ is an exponent of }  \HA \textrm{ with respect
  to } w\}
\end{equation}
is a basis for $\Fw(\HA)$, and we can write the series $\phi_v$ is as
in~\eqref{supp}, by the genericity of the parameters.

We wish to determine which elements of $\cB_w$ have an open domain of
convergence. This depends 
on the choice of the weight vector $w$.

\begin{theorem}
\label{teo:condnecysufparaconvdephi}
Let $\beta$ generic, $w$ a weight vector for $\HA$, and 
$\phi_v \in \cB_w$. There exists $\be \in \RR^{n-d}_{>0}$ such
that $\phi_v$ converges in $\cU_{w,\be}$ if and only if  
$|u| \geq 0$ for almost all $u\in\supp(\phi_v)$.
\end{theorem}

\begin{proof}
If $|u| \geq 0$ for almost all $u \in \supp(\phi_v)$, then $\phi_v$
converges on an open set $\cU_{w,\be}$ by Theorem~\ref{ultimolema}.

Now assume that there exist an infinite number of elements
$u\in\supp(\phi_v)$  such that $|u|<0$. Using the
description of $\supp(\phi_v)$  
from~\eqref{supp}, which applies when $\beta$ is generic, we can find 
$\nu \in \supp(\phi_v)$ such that $|\nu|<0$ and
$\{m\nu \mid  m\geq m_0\in\NN\}\subset\supp(\phi_v)$ for some
$m_0\in\NN$. 

Let $\psi$ be the subseries of $\phi_v$ whose terms are indexed by the set
$\{ m \nu \mid m \geq m_0 \in \NN\}$. The coefficient of $x^{v+m\nu}$
in $\psi$ is
\[
\frac{\prod_{\nu_i<0} \prod_{j=1}^{-m\nu_i}
  (v_i-j+1)}{\prod_{\nu_i\geq 0}\prod_{j=1}^{m\nu_i}(v_i+j)}\ ,
\]
which grows like 
$\lambda_m=\prod_{\nu_i<0} (-\nu_im)!/ \prod_{\nu_{i}>0} (\nu_im)!$
as $m$ goes to $\infty$. Since
$|\nu|<0$, $\lim_{m\rightarrow \infty} \lambda_m=\infty$.
Therefore $\psi$ cannot absolutely converge unless $x_i=0$ for some
$i$ such that $\nu_i>0$, and consequently $\phi_v$ does not have an
open domain of convergence.
\end{proof}

\begin{remark}
\label{remark:change of weight}
In this section we study convergence of Nilsson solutions of $\HA$ in
the direction of a weight vector $w$. We can change 
the point of view and fix a basis
$\{\gamma_{1},\ldots,\gamma_{n-d}\}$ of
$\ker_{\ZZ}(A)$; then our
results apply to any weight vector $w$ such that $\gamma_{i}\cdot w > 0$ for
$i=1,\ldots,n-d$.
\end{remark}

Since $\beta$ is generic, all the information necessary to compute the
Nilsson solutions of $\HA$ associated to a weight vector $w$
can be extracted from the top-dimensional
standard pairs of $\inow(I_{\rho(A)})$; the simplices appearing in
these standard pairs are the maximal simplices of the coherent
triangulation $\Delta_{(0,w)}$ of 
$\rho(A)$. These triangulations also control the possible regions of
convergence of basic Nilsson solutions of $\HA$ in the direction of $w$,
as a change in triangulation changes $\mathscr{U}_{w,\be}$.
There is an object that parametrizes all coherent triangulations of 
the configuration $\rho(A)$. This object is called
\emph{secondary fan} of $\rho(A)$, and was introduced by
Gelfand, Kapranov and Zelevinsky (see Chapter~7 of~\cite{GKZlibro}).

To construct the secondary fan of $\rho(A)$, we need a Gale dual
for $\rho(A)$, that is, a matrix $B \in \ZZ^{(n-d)\times(n+1)}$ whose columns 
form a basis for $\ker_{\ZZ}(\rho(A))$. Denote by $b_0,\ldots,b_n$ the rows
of $B$, and let
$\Delta$ be a coherent triangulation of $\rho(A)$.
For each maximal simplex
$\sigma\in\Delta$, we define a cone
\[
\cK_{\sigma}=\bigg\{\sum_{i\notin\sigma}\lambda_ib_i  \ \bigg| \ 
  \lambda_i \geq 0   \bigg\} .
\]
Note that the set $\{b_i \mid i\notin\sigma\}$ is linearly 
independent by Lemma~7.1.16 in~\cite{GKZlibro}, and therefore
$\cK_{\sigma}$ is full-dimensional.
Define
$
\cK_{\Delta} = \cap_{\sigma \in \Delta} \cK_{\sigma}
$.
Then $(w_0,w)\in \RR^{n+1}$ is such that $\Delta_{(w_0,w)} = \Delta$
if and only if $ (w_0,w) \cdot B$ belongs to the interior $\cK_{\Delta}^\circ$
of $\cK_{\Delta}$. The cones 
$\cK_{\Delta}$ for all coherent triangulations $\Delta$
of $\rho(A)$ are the maximal cones in a polyhedral fan, 
called the \emph{secondary fan of $\rho(A)$}
(see also Theorem~7.1.17 in~\cite{GKZlibro}).
%
%

Given $\beta$ generic, fix $w$ a weight vector for $\HA$.
Then the supports of the basic Nilsson solutions of $\HA$ in the
direction of $w$
can be described by means of the cones associated to maximal simplices
of the triangulation $\Delta_{(0,w)}$ of $\rho(A)$.
Indeed, if 
$(v_0,v)$ is the exponent associated to a standard pair
$(\partial^{\alpha},\sigma)$ of the monomial ideal $\inow(I_{\rho(A)})$, so
that 
$\sigma\in\Delta_{(0,w)}$ and $0\in\sigma$, we know that
the support of the dehomogenized series $\phi_v$ is
\[
\supp(\phi_v)=\{u\in \ker_{\ZZ}(A) \mid u_i+v_i\geq0 \ \ \forall
i\notin\sigma\}.
\]
Note that, as $0 \in \sigma$, the zeroth row of the Gale dual $B$ is
not present in this description.
Since the columns of $B$ span $\ker_{\ZZ}(\rho(A))$, 
the support of $\phi_v$ is naturally identified with 
\begin{equation}\label{eqn:support depends on cone}
\supp(\phi_v)= \left\{ (b_1\cdot \nu,\dots, b_n \cdot \nu) \mid  \nu
  \in\ZZ^{n-d} \text{ and }  \nu \cdot b_i 
  \geq-v_i, \ i\notin\sigma\right\}. 
\end{equation}

Theorem~\ref{teo:conv} gives a combinatorial condition for 
a series
$\phi_v$ associated to a cone $\cK_{\sigma}$ (that is, to a standard pair
$(\partial^{\alpha},\sigma)$) to have an open
domain of convergence. 
Note that
several series may be associated with a single cone, and some of them
may have open domains of convergence, while others do not, see
Example~\ref{ex:counterex-6.6}. 

Given $w$ a weight vector for $\HA$, let
$
\{(-|\gamma_1|,\gamma_1),\ldots,(-|\gamma_{n-d}|,\gamma_{n-d})\}
$
be a $\ZZ$-basis of $\ker_{\ZZ}(\rho(A))$ such 
that for any $i=1,\ldots,n-d$, we have $\gamma_i\cdot w > 0$.
The vectors in this basis are the columns of a Gale dual matrix of
$\rho(A)$, whose rows we denote by $b_0,\ldots,b_n$.

\begin{theorem}
\label{teo:conv}
Let $\beta$ be generic, $w$ a weight vector for $\HA$ and choose
a Gale dual of $\rho(A)$ as above.
Let $(v_0,v)$ be an exponent of $\rho(A)$ corresponding to a standard
pair $(\del^{\alpha},\sigma)$ of $\inow(I_{\rho(A)})$ that passes
through zero.
\begin{enumerate}[$\bullet$]
\item If $-b_0$ belongs to the interior of $\cK_{\sigma}$, then the series
$\phi_v$ has an open domain of convergence that contains $\cU_{w,\be}$
for some $\be \in \RR^{n-d}_{>0}$.
\item If $-b_0 \notin \cK_{\sigma}$, then the series $\phi_v$ does not
  have an open domain of convergence.
\end{enumerate}
\end{theorem}

\begin{proof}
%
%
For the first statement, suppose that $-b_0 \in \cK_{\sigma}^{\circ}$, so that 
$-b_0 = \sum_{i\notin \sigma} \lambda_ib_i$ with $\lambda_i > 0$. 
Then the polyhedron 
$\{ \nu \in \RR^{n-d} \mid \nu \cdot b_i \geq -v_i, i \notin \sigma, -b_0 \cdot \nu < 0
\}$
is bounded and can only contain finitely many points with
integer coordinates.
Using~\eqref{eqn:support depends on cone} and the fact that for $u =
(b_1 \cdot \nu, \dots, b_n\cdot \nu) \in \supp(\phi_v)$ we have:
\[
|u| = \sum_{i=1}^n b_i \cdot \nu = -b_0 \cdot \nu,
\]
we see that $|u| \geq 0$ for all but at most finitely many elements of
$\supp(\phi_v)$.  
Now apply Theorem~\ref{teo:condnecysufparaconvdephi} to conclude that
$\phi_v$ has an 
open domain of convergence of the desired form.

We now prove the second statement.
As $\{b_i \mid i\notin\sigma\}$ is a
basis of $\RR^{n-d}$, we can write $-b_0=\sum_{i\notin\sigma}\lambda_ib_i$.
Suppose that  $\lambda_{i_0}<0$ for some $i_0\notin\sigma$, and
consider the  infinite set
\[
\{ \nu \in \ZZ^{n-d} \mid \nu \cdot b_i = 0 \text{ for }
i \notin \sigma \cup \{ i_0 \} \textrm{ and } \nu \cdot b_{i_0} > 0 \}.
\]
For each element $\nu$ of this set, $(\nu \cdot
b_1,\dots, \nu\cdot b_n)$ is an element of $\supp(\phi_v)$, and the sum
of its coordinates is 
\[
\sum_{i=1}^n \nu \cdot b_i = \nu \cdot \big( \sum_{i=1}^n b_i \big) =
\nu \cdot (-b_0) = \nu \cdot (\sum_{i \notin \sigma} \lambda_i b_i) =
\lambda_{i_0} (\nu \cdot b_{i_0}).
\]
Thus, $\supp(\phi_v)$ has an infinite subset consisting of vectors
whose coordinate sum is negative, and by
Theorem~\ref{teo:condnecysufparaconvdephi}, $\phi_v$ does not have an
open domain of convergence.
\end{proof}

\begin{example}
\label{ex:counterex-6.6}
If $-b_0$ is in the boundary of $\cK_{\sigma}$, a series arising
from $\sigma$ may or may not have an open domain of convergence,
depending on the specific standard pair it belongs to. For instance,
consider:
\[
A = \left[ \begin{array}{rrrr}
0 & 1 & 0 & -1 \\
1 & 0 & 2 & 4
\end{array} \right].
\]
Then the convex hull of $\rho(A)$ is a triangle of normalized volume
$4$. For the weight $(0,w) =$ $(0,1,1,3,3)$, the corresponding triangulation of
$\rho(A)$ has only one maximal simplex, $\conv(\rho(A))$ itself.
In this case,
\[
I_{\rho(A)} = \langle
\del_3^2-\del_2\del_4,\del_1^2-\del_0\del_3 \rangle 
\quad \text{and} \quad
\inow(I_{\rho(A)}) = \langle
\del_0\del_3,\del_1^4,\del_1^2\del_3,\del_3^2\rangle.
\] 
The monomial ideal $\inow(I_{\rho(A)})$ has four top-dimensional standard pairs:
\[ (1, \{0,2,4\}), \quad (\del_1,\{0,2,4\}), 
\quad (\del_1^2, \{0,2,4\}) \quad \text{and} \quad (\del_1^3,\{0,2,4\}). 
\]
A Gale dual of $\rho(A)$ is 
\[
B = \left[ \begin{array}{rr}
1 & 0 \\
-2 & 0 \\
0 & -1 \\
1 & 2 \\
0 & -1
\end{array}\right]
\]
and we have $-b_0 = (1/2) b_1$, so that $-b_0$ is in the boundary of
$\cK_{\sigma}$ for $\sigma = \{ 0,2,4 \}$, the unique maximal simplex in the
triangulation of $\rho(A)$ induced by $(0,1,1,3,3)$.

For a given generic $\beta$ and $0 \leq j \leq 3$, let $\phi_j$ denote
the logarithm-free series solution of $\HA$ associated to the standard
pair $(\del_1^j,\sigma)$. Then
\[
\supp(\phi_j) = 
\left\{ u = (-2 \nu_1, -\nu_2, \nu_1 + 2 \nu_2,-\nu_2 ) \left| 
\begin{array}{l}
\nu=(\nu_1,\nu_2) \in \ZZ^2, \\
-2 \nu_1 \geq -j, \;\;
\nu_1+2\nu_2 \geq 0
\end{array}
\right\}
\right.
.
\]
For $u \in \supp(\phi_j)$ as above, $|u| = -\nu_1 \geq
-j/2$. Consequently, if $j=0,1$, $|u| \in \ZZ$ implies that $|u| \geq
0$ for all $u \in \supp(\phi_j)$, and therefore $\phi_j$ has an open
domain of convergence by 
Theorem~\ref{teo:condnecysufparaconvdephi}. 

If $j=2,3$, consider $\nu=(1,k)$ for $k \in \NN$. Then
\[
-2\nu_1  = -2 \geq -j \quad \text{and} \quad \nu_1 +2\nu_2 =1+2k \geq
0 \quad \text{for} \; k \in \NN. 
\]
This means that 
$u=(-2 \nu_1, -\nu_2, \nu_1+2\nu_2,-\nu_2 )=(-2,-k,1+2k,-k) \in \supp(\phi_j)$ for
$k\in \NN$. But then $|u| = -1$ for infinitely many elements of
$\supp(\phi_j)$, and therefore $\phi_j$ does not have an open domain
of convergence.

Geometrically, consider the polyhedron 
\[
Q_j = \{
\nu \in \RR^2 \mid -2 \nu_1 \geq -j, \;\;
\nu_1 -2 \nu_2 \geq 0
\},
\]
which is a translate of a cone. Let $L$ be the line orthogonal to 
$-b_0 = (-1,0)$.
Since $L$ contains one of the faces of $Q_0$,
there are no points in this polyhedron whose dot product with $-b_0$ is strictly
negative.
If $j = 1, 2, 3$, then the set $ Q_j \cap \{ \nu \in \RR^2 \mid -b_0
\cdot \nu < 0 \}$ in unbounded. However, when $j=1$, this set contains
no integer points. When $j=2,3$,  $ Q_j \cap \{ \nu \in \RR^2 \mid -b_0
\cdot \nu < 0 \}$
contains infinitely many integer points. 

This example also illustrates the fact that, in order to determine 
whether a basic Nilsson solution of $H_A(\beta)$ in the direction 
of $w$ has an open domain of convergence, knowledge of the triangulation 
of $\rho(A)$ induced by $(0,w)$ is not sufficient.
For instance, note that $(1,1,2,1)$ is a perturbation of $(1,1,1,1)$, so all Nilsson 
solutions of $H_A(\beta)$ in the direction of $(1,1,2,1)$ have an open domain 
of convergence by Corollary~\ref{coro:convergence-of-solutions-from-111}.
However, $(0,1,1,2,1)$ and $(0,1,1,3,3)$ induce the same triangulation of $\rho(A)$, 
and we already know that there are Nilsson solutions of $H_A(\beta)$ in the direction 
of $(1,1,3,3)$ which do not have open domains of convergence.
\end{example}

As a consequence of Theorem~\ref{teo:conv}, we have combinatorial
bounds for the dimension of the space
of convergent Nilsson solutions of $\HA$ in the direction of $w$, that
holds for generic parameters $\beta$.

\begin{corollary}
\label{dimconv}
Let $\beta$ be generic, and $w$ a weight vector for $\HA$. Then
\[
\sum_{\sigma\in T}\vol(\sigma) \geq \dim_{\CC}(\{\phi \in \Fw(\HA) \text{ convergent }\})
\geq
\sum_{\sigma\in T_\circ}\vol(\sigma)
 \, ,
\]
where 
\[
T_\circ=\left\{
\begin{array}{l}
\sigma {\text{ facet }} \\
{\text{of }}\Delta_{(0,w)}
\end{array}
\left|
\begin{array}{l} 
0\in\sigma \\
-b_0\in\cK_{\sigma}^{\circ} 
\end{array}
\right.
\right\};
\quad
T=\left\{
\begin{array}{l}
\sigma {\text{ facet}}\\
{\text{of }}\Delta_{(0,w)} 
\end{array}
\left|
\begin{array}{l}
0\in\sigma \\
-b_0\in\cK_{\sigma} 
\end{array}
\right.
\right\}
. 
\]
\qed
\end{corollary}

The second part of Theorem~\ref{teo:conv} can be restated in a
more combinatorial fashion.

\begin{corollary}\label{coro: dibujitos}
Let $\beta$ be generic, and $w$ a weight vector for $\HA$.
If an element of the set $\cB_w$ from~\eqref{eqn:Bw} has an open domain
of convergence, then
its associated maximal simplex in the triangulation $\Delta_{(0,w)}$
of $\rho(A)$ also belongs to a coherent triangulation of $\rho(A)$
defined by $(0,w')$, where $w'$ is a perturbation of $(1,\dots,1)$.
\end{corollary}

\begin{proof}
Using Theorem~\ref{teo:conv} and its notation, we have that, if
an element of the set $\cB_w$ from~\eqref{eqn:Bw} has an open domain
of convergence, then
$-b_0 = (0, 1, \dots, 1) \cdot B \in \cK_{\sigma}$. Therefore, there is a
perturbation $w'$ of  $(1,\dots,1)$ with $\sum_{i=1}^n w_i'b_i \in \cK_{\sigma}^\circ$.
It follows that $\sigma$ is a maximal simplex in the coherent triangulation of
$\rho(A)$ defined by $(0,w')$.
\end{proof}

\begin{example}\label{ejem: dibujitos}
Corollary~\ref{coro: dibujitos} allows us to decide
by inspecting the triangulations of $\rho(A)$, whether $\HA$ has
Nilsson solutions in the direction of a weight vector that do not have an
open domain of convergence.
Take for instance 
\[
A=\left[\begin{array}{ccccc}
2 & 0 & 1 & 2 & 2 \\
0 & 2 & 2 & 1 & 2\\
\end{array}\right]
\]
and consider the following coherent triangulations of $\rho(A)$, or, 
equivalently, the coherent triangulations of $A \cup \{0\}$.
Note that the triangulations $\Delta_i, \ i=1,\dots,4$ appearing in
Figure~\ref{figure 1}
\begin{figure}
 \begin{tabular}{cc}
  \includegraphics[scale=0.7]{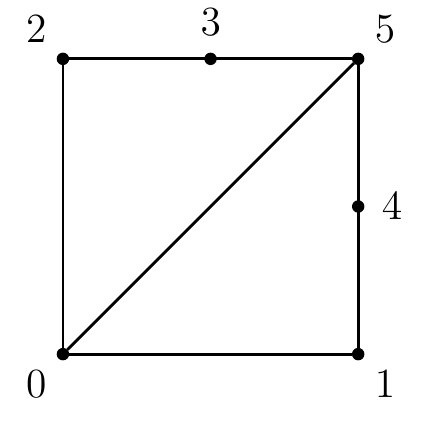}   & \includegraphics[scale=0.7]{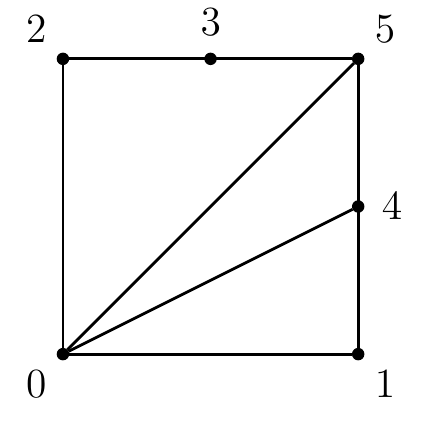}
    \\
     {\small $\Delta_1=\left\{\{0,1,5\},\{0,2,5\}\right\}$}   &
{\small $\Delta_2=\{\{0,1,4\},\{0,4,5\},\{0,2,5\}\}$}\\                 
  \includegraphics[scale=0.7]{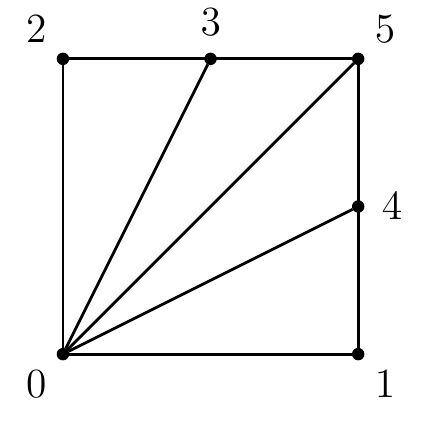}                                &
\includegraphics[scale=0.7]{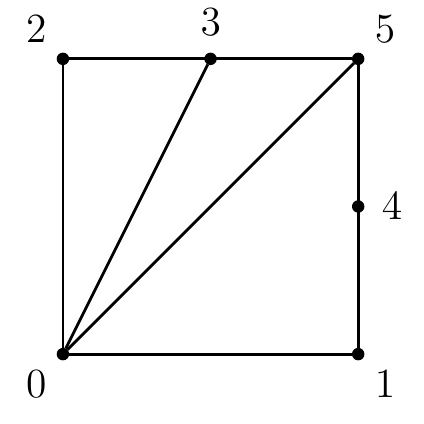}     \\
{\small $\Delta_3=\{\{0,1,4\},\{0,4,5\},\{0,3,5\},\{0,2,3\}\}$}   &
{\small $\Delta_4=\{\{0,1,5\},\{0,3,5\},\{0,2,3\}\}$ }               
 \end{tabular}
\caption{The coherent triangulations of $\rho(A)$ corresponding to
  perturbations of the vector $(0,1\dots,1)$ for Example~\ref{ejem: dibujitos}.}
\label{figure 1}
\end{figure} 
\begin{figure}
\begin{tabular}{c}
 \includegraphics[scale=0.7]{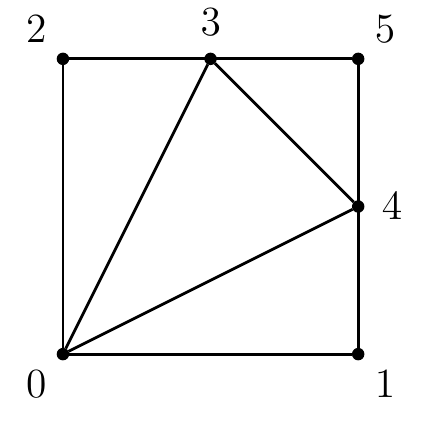}\\
{\small $\Delta_5=\{\{0,1,4\},\{0,3,4\},\{0,2,3\},\{3,4,5\}\}$}
\end{tabular}
\caption{A coherent triangulation of $\rho(A)$ 
from Example~\ref{ejem: dibujitos}.}
\label{figure 2}
\end{figure}
are all the triangulations 
induced by vectors $(0,w')$ with $w'$
a perturbation of $(1,\dots,1)$.
Now consider the triangulation $\Delta_5$ drawn in Figure~\ref{figure 2}. 
The simplex $\{0,3,4\}$ belongs to $\Delta_5$ and passes through zero,
but does not appear in any 
triangulation of $\rho(A)$ induced by a perturbation of $(0,1,\dots,1)$; therefore,
Corollary~\ref{coro: dibujitos} ensures that the corresponding Nilsson solutions
of $\HA$ (for generic $\beta)$ do not have open domains of convergence.
\end{example}

\section{The irregularity of $\HA$ via its Nilsson solutions}
\label{sec:Hotta}

In this section, we give an alternative proof 
of Corollary~3.16 in~\cite{SW} using our study of Nilsson solutions of
$\HA$. We assume as in~\cite{SW} that the columns of $A$ span a
strongly convex cone. 

In Theorem~\ref{thm:image of rho} we computed the image of $\rho$.
There is one case in which this map is guaranteed to be onto.

\begin{proposition}
\label{cor:surj-case}
Suppose that $w$ is a perturbation of $(1,\dots,1)$. Then the map
$\rho$ is surjective. 
\end{proposition}

\begin{proof}
First note that none of the minimal generators of the monomial ideal
$\inow(I_{\rho(A)})$ 
have $\del_0$ as a factor. This implies that, if $u \in \NN^{n+1}$ and
$\del^u \notin \inow(I_{\rho(A)})$, then $\del_0^k \del^u \notin
\inow(I_{\rho(A)})$ for all $k \in \NN$. We conclude that all the
standard pairs of $\inow(I_{\rho(A)})$ pass through
zero. Theorem~\ref{thm:image of rho} completes the proof.
\end{proof}

We wish to find weight a vector $w$ for $\HA$ for which $\rho$
is not surjective. We require the following statement.

\begin{lemma}
\label{lematriang}
Let $A\in\ZZ^{d\times n}$ of full rank $d$ whose columns span a
strongly convex cone. If the row span of $A$ does not contain the
vector $(1, \dots, 1)$,
there exists $w\in\RR_{> 0}^n$ such that the coherent triangulation
$\Delta_{(0,w)}$ of $\rho(A)$ has a maximal simplex
that does not pass through zero. 
Given $\beta \in \CC^d$, the vector $w$ can be chosen to be a weight
vector for $\HA$. 
\end{lemma}

\begin{proof}
We use the description of the secondary fan of $\rho(A)$ from
Section~\ref{sec:convergence}. 

Let $B$ be a Gale dual matrix of $\rho(A)$ with rows
$b_0,\dots,b_n$. Since $(1, \dots, 1)$ is 
not in the rowspan of $A$, the zeroth row of $B$ is nonzero.
Because $B$ has full rank $n-d$, we can choose 
$\sigma \subset \{1,\dots n\}$ of cardinality $d+1$ such that 
$\{ b_i \mid i \not \in \sigma \}$ is linearly independent. 

The assumption that the columns $a_1,\dots,a_n$ of $A$ span a strongly
convex cone means that there exists a vector $h \in \RR^d$ such
that $h \cdot A$ is coordinatewise positive. As $\rho(A) \cdot B = 0$,  
$\sum_{i=1}^n (h \cdot a_i) b_i = 0$.
 
Choose $w \in \RR_{> 0}^n$ and positive real
$\lambda_i$ for $i \notin \sigma$, such that
\[
w_i + \lambda_0 = h \cdot a_i \; \text{ for } i \in \sigma\, , \;
\text{ and } w_i + \lambda_0 - \lambda_i = h \cdot a_i \; \text{ for } i
\notin \sigma \cup \{ 0 \}.
\]
There is enough freedom in the choice of $w$ that we may assume that
$(0,w)$ induces a triangulation $\Delta_{(0,w)}$ of $\rho(A)$ and not merely a
subdivision. This also implies that $w$ can be chosen a weight vector
for $\HA$, if $\beta \in \CC^d$ is given.

We claim that $(0,w)\cdot B \in \cK_{\sigma}^\circ$. This implies that
$\sigma$ is a maximal simplex in $\Delta_{(0,w)}$ which does not pass
through zero.

To prove the claim, note that 
\[
\sum_{i \in \sigma} (w_i+\lambda_0) b_i + 
\sum_{i \notin \sigma \cup \{ 0 \} } (w_i+\lambda_0-\lambda_i) b_i = 
\sum_{i=1}^n (h \cdot a_i) b_i = 0.
\]
Then 
\[
\sum_{i=1}^n w_i b_i 
= 
\sum_{i \notin \sigma \cup \{ 0 \}} (\lambda_i-\lambda_0) b_i - 
\lambda_0 \sum_{i \in \sigma} b_i 
= 
\sum_{i \notin \sigma \cup\{ 0 \} } \lambda_i b_i - \lambda_0 \sum_{i=1}^n b_i
= 
\sum_{i \notin \sigma} \lambda_i b_i ,
\]
where the last equality follows from $-b_0 = \sum_{i=1}^nb_i$. But
then, $\lambda_i > 0$ for $i \notin \sigma$ implies that
$(0,w)\cdot B \in \cK_{\sigma}^\circ$, which is what we wanted.
\end{proof}

The hypothesis that the columns of $A$ span a strongly convex cone
cannot be removed from
Lemma~\ref{lematriang}, as the following example shows.

\begin{example}\label{ejem: lematriang}
Let $A=[-1\; 1]$. Then
\[
\rho(A)=\left[\begin{array}{crc}
 1 & 1 & 1 \\
 0 & -1 & 1\\
\end{array}\right] ,\ \textrm{and choose } \ 
B=\left[\begin{array}{r} -2\\ 1 \\1 
\end{array}\right].
\]
There are only two coherent triangulations of $\rho(A)$, namely
\[
\Delta_1=\{\{-1,1\}\} \textrm{ and } \Delta_2=\{\{-1,0\},\{0,1\}\}. 
\]
Their
corresponding cones in the secondary fan are 
$\cK_{\Delta_1}=\RR_{\leq 0}$ and 
$\cK_{\Delta_2}=\RR_{\geq 0}$. For any vector
$w=(w_{-1},w_1)\in\RR_{> 0}^2$, the number 
$0\cdot b_0+w_{-1}\cdot b_{-1}+w_1\cdot b_1=w_{-1}+w_1$ 
belongs to the cone $\cK_{\Delta_2}^\circ=\RR_{> 0}$ and
consequently $w$ always induces a triangulation of $\rho(A)$ all of whose
maximal simplices pass through zero.
\end{example}

\begin{proposition}
\label{coro:non-surj}
Assume that the columns of $A$ span a strongly convex cone.
If the rational rowspan of $A$ does not contain the 
vector $(1, \dots, 1)$, there exists a
weight vector $w$ such that the linear map  
\[
\rho:\Fw(\HA) \rightarrow {\mathscr{N}}_{(0,w)}(\HrA)
\]
is not surjective, where $\beta_0$ be a homogenizing value for $A$,
$\beta$ and $w$.
\end{proposition}

\begin{proof}
Use Lemma~\ref{lematriang} to pick $w$ so that the triangulation
$\Delta_{(0,w)}$ of $\rho(A)$ has a maximal simplex
that does not pass through zero.
Then 
$\inow(I_{\rho(A)})$ has top-dimensional standard pairs that do not
pass through zero. Choose a homogenizing value $\beta_0$ for $A$,
$\beta$ and 
$w$, and let $(v_0,v)$
be a fake exponent of $\HrA$ corresponding to such a standard pair
(fake exponents associated to top-dimensional standard pairs always
exist). In 
particular, $v_0 \in \NN$. If $(v_0,v)$ has minimal negative support,
it is an exponent of $\HrA$ corresponding to a standard pair that does
not pass through zero, and the associated logarithm-free solution
$\phi_{(v_0,v)}$ of $\HrA$ cannot belong to the image of $\rho$ by
Theorem~\ref{thm:image of rho}.
If $(v_0,v)$ does not have minimal negative
support, the argument of Proposition~3.4.16 in~\cite{SST} produces an element
$ (v_0',v')\in ((v_0,v)+\ker_{\ZZ}(A))$
whose negative support is minimal and 
strictly contained in the negative support of $(v_0,v)$, so that $v'_0$ is still
a non negative integer. Thus, the standard pair corresponding to
$(v_0',v_0)$ cannot 
pass through zero, and $\phi_{(v_0',v')}$ is not in the image of $\rho$.
\end{proof}

The following result is due to Berkesch (see Theorem~7.3 in~\cite{Christine's-thesis}).

\begin{theorem}
\label{thm:rank-lifting}
If the cone over the columns of $A$ is strongly convex and $\beta_0$ is generic,
\[\rank(\HA) = \rank(\HrA).\]
\end{theorem}

We now give a proof for Theorem~\ref{ultimoteorema}. 

\begin{proof}[Proof of Theorem~\ref{ultimoteorema}]
Given $\beta_0$ a homogenizing value for $A$,
$\beta$ and $w$,
Proposition~\ref{cor:surj-case} states that
the spaces $\Fw(\HA))$ and $\mathscr{N}_{(0,w)}(\HrA))$ are isomorphic.
We may further assume that $\beta_0$ is sufficiently generic that
Theorem~\ref{thm:rank-lifting} holds.

As $I_\rho(A)$ is homogeneous, $\HrA$ is regular holonomic, 
and~\cite[Corollary~2.4.16]{SST} implies that  
$\dim_{\CC}(\mathscr{N}_{(0,w)}(\HrA)) = \rank (\HrA)$. So, 
\begin{align*}
\dim_{\CC}(\Fw(\HA))) & = \dim_{\CC}(\mathscr{N}_{(0,w)}(\HrA))) \\
& = \rank(\HrA)=\rank(\HA),
\end{align*} 
where the last equality is by Theorem~\ref{thm:rank-lifting}.

Since $w$ is a perturbation of $(1,\dots,1)$, if $\phi$ is a
basic Nilsson solution of $\HA$ in the direction of $w$, then $|u|\geq 0$
for all $u \in \supp(\phi)$. Therefore we can use
Theorem~\ref{ultimolema} to find $\be \in \RR_{>0}^{n-d}$ such that
all basic Nilsson solutions 
of $\HA$ in the direction of $w$ converge on $\cU_{w,\be}$.
\end{proof}

We are finally ready to show that, in the case when the columns of $A$
span a strongly convex cone, non homogeneous $A$-hyper\-geo\-metric
systems are irregular for all $\beta$, thus
generalizing the argument in Theorem~2.4.11 from~\cite{SST}, 
and providing an alternative proof of Corollary~3.16 from~\cite{SW}.

\begin{theorem} 
\label{coro:lower rank initial}
Assume that the columns  of $A$ span a strongly convex cone and $I_A$ is
not homogeneous. Then $\HA$ is not regular
holonomic for any $\beta \in \CC^d$.
\end{theorem}

\begin{proof}

Choose $w$ a weight vector for $\HA$ as in
Proposition~\ref{coro:non-surj} and $\beta_0$ a homogenizing value
for $A$, $\beta$ and $w$.
Then
\begin{align*}
\dim_{\CC}(\Fw(\HA))    & < \dim_{\CC}({\mathscr{N}}_{(0,w)}(\HrA)) \\
                  & = \rank(\HrA) \\
                  & = \rank(\HA).
\end{align*}
The equality in the second line
follows from Corollary~2.4.16 in~\cite{SST} because the system $\HrA$ is
regular holonomic, as $I_{\rho(A)}$ is homogeneous.
The last equality is by
Theorem~\ref{thm:rank-lifting} (we may have to make $\beta_0$ more
generic for this result to hold, but this does not affect
Proposition~\ref{coro:non-surj}). Now Corollary~2.4.16 in~\cite{SST} and
$\dim_{\CC}(\Fw(\HA)) < \rank(\HA)$ imply that $\HA$ is
not regular holonomic. 
\end{proof}


\def\cprime{$'$}

\end{document}